\documentclass[12pt,a4paper,reqno]{amsart}
\usepackage[a4paper,left=30mm,right=30mm,top=36mm,bottom=36mm]{geometry}

\newtheorem{definition}{Definition}[section] 
\newtheorem{remark}{Remark}[section]    
\newtheorem{theorem}{Theorem}[section]
\newtheorem{conjecture}{Conjecture}[section]        
\newtheorem{question}{Question}[section]        
   
\newtheorem{lemma}{Lemma}[section]	
\newtheorem{corollary}{Corollary}[section]   

\numberwithin{equation}{section}

\usepackage{amssymb}
\usepackage{amsmath}
\usepackage{graphicx}
\usepackage{float}
\usepackage[utf8]{inputenc}
\usepackage[english]{babel}
\usepackage{color}

\newcommand{\R}{\mathbb{R}}
\newcommand{\T}{\mathbb{T}}
\newcommand{\N}{\mathbb{N}}
\newcommand{\Z}{\mathbb{Z}}
\newcommand{\eps}{\varepsilon}
\newcommand{\fhi}{\varphi}
\newcommand{\tr}{\mathrm{tr}}

\def\calO{\mathcal{O}}
\def\calS{\mathcal{S}}

\newcommand{\RNum}[1]{\uppercase\expandafter{\romannumeral #1\relax}}

\usepackage{hyperref}

\begin{document}

\title[Characterizations of diffusion matrices in homogenization]{Characterizations of diffusion matrices in homogenization of elliptic equations in nondivergence-form}

\author[X. Guo]{Xiaoqin Guo}
\address[Xiaoqin Guo]{
Department of Mathematical Sciences,
4314 French Hall,
University of Cincinnati,
Cincinnati, OH 45221.}
\email{guoxq@ucmail.uc.edu}

\author[T. Sprekeler]{Timo Sprekeler}
\address[Timo Sprekeler]{
Department of Mathematics,
National University of Singapore, 
10 Lower Kent Ridge Road, 
Singapore 119076.}
\email{timo.sprekeler@nus.edu.sg}

\author[H. V. Tran]{Hung V. Tran}
\address[Hung V. Tran]{
Department of Mathematics, 
University of Wisconsin Madison, 
Van Vleck Hall, 480 Lincoln Drive, 
Madison, Wisconsin 53706, USA.}
\email{hung@math.wisc.edu}

\date{\today}

\begin{abstract}
We characterize diffusion matrices that yield a $L^{\infty}$ convergence rate of $\calO(\eps^2)$ in the theory of periodic homogenization of linear elliptic equations in nondivergence-form. Such type-$\eps^2$ diffusion matrices are of particular interest as the optimal rate of convergence in the generic case is only $\calO(\eps)$. First, we provide a new class of type-$\eps^2$ diffusion matrices, confirming a conjecture posed in \cite{GT20}.
Then, we give a complete characterization of diagonal diffusion matrices in two dimensions and a systematic study in higher dimensions.
\end{abstract}

\maketitle

\section{Introduction}

In this paper, we are concerned with the prototypical linear elliptic nondivergence-form problem 
\begin{align}\label{ueps problem}
\left\{\begin{aligned}-A\left(\frac{\cdot}{\eps}\right):D^2 u^{\eps} &= f& &\text{in }\Omega,\\
u^{\eps} &= g&  &\text{on }\partial\Omega,\end{aligned}\right.
\end{align}
posed on a bounded smooth domain $\Omega\subset \R^n$, where $\eps > 0$ is a parameter (considered to be small), $f\in C(\Omega)$ and $g\in C(\bar{\Omega})$ are continuous scalar functions of regularity
\begin{align}\label{f,g ass}
f\in W^{3,q}(\Omega),\quad g\in W^{5,q}(\Omega)\quad\text{for some }q>n,
\end{align}
and $A:\R^n\rightarrow \R^{n\times n}_{\mathrm{sym}}$ is a symmetric matrix-valued map which is assumed to be H\"{o}lder continuous, $\Z^n$-periodic and uniformly elliptic, i.e.,
\begin{align*}
A\in C^{0,\alpha}(\T^n;\calS^n_+)
\end{align*}   
for some $\alpha\in (0,1)$. Throughout this work, $\T^n:=\R^n / \Z^n$ is the flat $n$-dimensional torus and $\calS^n_+ \subset \R^{n\times n}_{\mathrm{sym}}$ the set of symmetric positive definite $n\times n$ matrices. 

\smallskip

We prove that if $A$ is of the form $A(y)=C+a(y)M$ for $y\in \R^n$, where $C, M\in \R^{n\times n}_{\mathrm{sym}}$ are constant symmetric matrices and $a\in C^{0,\alpha}(\T^n)$, then $(u^{\eps})_{\eps>0}$ converges to the homogenized solution in the $L^{\infty}(\Omega)$-norm with convergence rate $\calO(\eps^2)$ as $\eps\searrow 0$, which confirms a conjecture in \cite{GT20} (see Theorem \ref{Thm: Main}). 

\smallskip

Let $Y := [0,1]^n$ denote the unit cell in $\R^n$. 
Introducing the invariant measure $r\in C^{0,\alpha}(\T^n)$ (see e.g., \cite{AL89,ES08,BKR01,Sjo73}) as the unique solution to the periodic problem
\begin{align}\label{r problem}
-D^2:(rA) = 0\quad\text{in }Y,\qquad r \text{ is }Y\text{-periodic},\qquad r>0,\qquad \int_Y r = 1,
\end{align}
it is well known that the sequence of solutions $(u^{\eps})_{\eps>0}$ to \eqref{ueps problem} converges uniformly on $\bar{\Omega}$ to the solution $u$ of the homogenized/effective problem
\begin{align}\label{u problem}
\left\{\begin{aligned}-\bar{A}:D^2 u &= f& &\text{in }\Omega,\\
u &= g&  &\text{on }\partial\Omega,\end{aligned}\right.
\end{align}
where $\bar{A}\in \calS^n_+$ denotes the effective coefficient which is given by
\begin{align*}
\bar{A}:=\int_Y rA; 
\end{align*}
see e.g., \cite{BLP11,JKO94}. For $k,l\in \{1,\dots,n\}$, the (k,l)-th entry $\bar{a}_{kl}$ of $\bar{A}$ can be equivalently characterized as the unique value such that the (k,l)-th cell problem
\begin{align}\label{vij problem}
-A:D^2 v^{kl} = a_{kl}-\bar{a}_{kl}\quad\text{in }Y,\qquad v^{kl} \text{ is }Y\text{-periodic},\qquad \int_Y v^{kl} = 0
\end{align}
admits a unique solution $v^{kl}\in C^{2,\alpha}(\T^n)$, a so-called corrector function. 

The objective of this paper is to obtain new insights on the classes of type-$\eps^2$ and type-$\eps$ matrix-valued maps which we define as follows.
\begin{definition}[type-$\eps^2$ and type-$\eps$ diffusion matrices]\label{Def: cg cb}
Let $A\in C^{0,\alpha}(\mathbb{T}^n;\calS^n_+)$ for some $\alpha\in (0,1)$. For $j,k,l\in \{1,\dots,n\}$, we define the value $c_j^{kl}(A)\in \R$ by
\begin{align}\label{eq:def-cjkl}
c_j^{kl}(A):= \int_{Y} rAe_j\cdot \nabla v^{kl},
\end{align} 
where $r\in C^{0,\alpha}(\mathbb{T}^n)$ and $v^{kl}\in C^{2,\alpha}(\mathbb{T}^n)$ denote the solutions to \eqref{r problem} and \eqref{vij problem}, and $e_j\in \R^n$ denotes the vector with entries $(e_j)_i:=\delta _{ij}$ for $i\in \{1,\dots,n\}$, respectively. We call $A$ a type-$\eps^2$ diffusion matrix (or type-\RNum{2} diffusion matrix) if there holds
\begin{align*}
C_{jkl}(A):=c_j^{kl}(A)+c_k^{jl}(A)+c_l^{jk}(A) = 0\qquad\forall j,k,l\in \{1,\dots,n\}.
\end{align*}
Otherwise, we call $A$ a type-$\eps$ diffusion matrix (or type-\RNum{1} diffusion matrix).
\end{definition}
In the literature, $(c_j^{kl}(A))_{1\leq j,k,l\leq n}$ is commonly referred to as the third-order homogenized tensor; thanks to \cite{Arm22}. Up to a multiplicative constant, $(C_{jkl}(A))_{1\leq j,k,l\leq n}$ is the symmetric part of this third-order tensor.

Let us note that if $c_j^{kl}(A)=0$ for all $1\leq j,k,l\leq n$, then $A$ is type-$\eps^2$. Further, if $A$ is diagonal, then $v^{kl} \equiv 0$ for $k \neq l$, and thus, $A$ is type-$\eps^2$ if and only if $c_j^{kl}(A)=0$ for all $1\leq j,k,l\leq n$. We point out that there is a typo in the definition of ``c-good" (type-$\eps^2$) and ``c-bad" (type-$\eps$) diffusion matrices in \cite{GTY20,ST21}, which was intended to be as, and should be replaced by, Definition \ref{Def: cg cb}.

Introducing $z$ as the unique solution to the problem
\begin{align}\label{z problem}
\left\{\begin{aligned}-\bar{A}:D^2 z &= -\sum_{j,k,l=1}^n c_j^{kl}(A)\, \partial_{jkl}^3 u& &\text{in }\Omega,\\
z &= 0&  &\text{on }\partial\Omega,\end{aligned}\right.
\end{align}
it is well-known that 
\begin{align}\label{Linfty bound}
\left\|u^{\eps} - u + 2\eps z\right\|_{L^{\infty}(\Omega)} = \calO(\eps^2)\quad\text{as }\eps \searrow 0;
\end{align}
see e.g., \cite{GTY20,ST21} (note \cite{ST21} assumes $g\equiv 0$, but can be extended to $f,g$ satisfying \eqref{f,g ass}). Therefore, the classification of maps $A\in C^{0,\alpha}(\mathbb{T}^n;\calS^n_+)$ into type-$\eps^2$ and type-$\eps$ diffusion matrices is a classification by the corresponding optimal convergence rate for the convergence $u^{\eps}\rightarrow u$ in the $L^{\infty}(\Omega)$-norm as $\eps \searrow 0$.
\begin{remark}[Optimal rates of convergence]
Let $A\in C^{0,\alpha}(\mathbb{T}^n;\calS^n_+)$ for some $\alpha \in (0,1)$. By \eqref{Linfty bound} and \eqref{z problem}, we have the following assertions.
\begin{itemize}
\item[(i)] If $A$ is type-$\eps$, then we have that
\begin{align}\label{c-bad O(eps)}
\left\|u^{\eps} - u \right\|_{L^{\infty}(\Omega)} = \calO(\eps)\quad\text{as }\eps \searrow 0
\end{align}
for any choice of $f,g$, and the convergence rate $\calO(\eps)$ in \eqref{c-bad O(eps)} is optimal in general, i.e., there is a choice of $f,g$ for which $\left\|u^{\eps} - u \right\|_{L^{\infty}(\Omega)} = o(\eps)$ fails.
\item[(ii)] If $A$ is type-$\eps^2$, then the solution to \eqref{z problem} is $z\equiv 0$ and thus,
\begin{align}\label{c-bad O(eps^2)}
\left\|u^{\eps} - u \right\|_{L^{\infty}(\Omega)} = \calO(\eps^2)\quad\text{as }\eps \searrow 0
\end{align}
for any choice of $f,g$. Further, if $A$ is not constant, then the convergence rate $\calO(\eps^2)$ in \eqref{c-bad O(eps^2)} is optimal in general, i.e., there is a choice of $f,g$ for which $\|u^{\eps}-u\|_{L^{\infty}(\Omega)} = o(\eps^2)$ fails.
\end{itemize}
\end{remark}

Let us note that optimal convergence rates are not only important in analysis, but can also be used for deriving optimal error bounds in the numerical homogenization of elliptic equations in nondivergence-form. The development of numerical homogenization schemes for \eqref{ueps problem} and more generally, for fully-nonlinear equations of nondivergence structure is an active area of research; see e.g., \cite{CSS20,FO09,GSS21,KS21,CM09,FO18} and the references therein. For some results on convergence rates and error estimates in periodic homogenization of divergence-form equations; see \cite{Gri06,KLS12,KLS14,MV97,OV07,Sus13}.

It is known that the set $\{A\in C^{\infty}(\mathbb{T}^n;\calS^n_+): A\text{ is type-$\eps$}\}$ is open and dense in $C^{\infty}(\mathbb{T}^n;\calS^n_+)$ when $n\geq 2$; see the perturbation argument given in \cite{GTY20}. An explicit example of a type-$\eps$ diffusion matrix in dimension $n=2$ is the map $A\in C^{\infty}(\mathbb{T}^2;\calS^2_+)$ given by
\begin{align}\label{cbad from ST}
\begin{split}
A(y)&:=\frac{1}{r(y)}\, \mathrm{diag}\left(1-\frac{1}{2}\sin(2\pi y_1)\sin(2\pi y_2),1+\frac{1}{2}\sin(2\pi y_1)\sin(2\pi y_2)\right),\\
r(y)&:= 1+\frac{1}{4}(\cos(2\pi y_1)-2\sin(2\pi y_1))\sin(2\pi y_2)
\end{split}
\end{align}
for $y=(y_1,y_2)\in \R^2$; see \cite{ST21}. In fact, as evident in \eqref{Linfty bound} and \eqref{z problem}, the main source of the homogenization error comes from the averaged ``correlation"  $c_{j}^{kl}(A)$ between the gradients $\nabla v^{kl}$ of correctors and the field $A$ of the coefficients. It may seem natural to expect that high dimensionality offers faster decorrelation and hence yields at least the same rate for the homogenization, if not better.  Surprisingly, simulations show that in dimension $n=2$, all diagonal diffusion matrices with constant trace are type-$\eps^2$, and this is not true in dimensions $n\geq 3$. This has lead to the following conjecture in \cite{GT20}.

\begin{conjecture}[Conjecture 1 in \cite{GT20}]\label{Conj}
Assume that $A\in C^2(\T^2;\calS^2_+)$ is of the form
\begin{align*}
A(y) = \mathrm{diag}(a_1(y),a_2(y))\quad\text{for}\quad y\in \R^2,
\end{align*}
where $a_1,a_2\in C^2(\T^2;(0,1))$ and $a_1+a_2 \equiv 1$. Then, we conjecture that 
\begin{align*}
\left\|u^{\eps} - u \right\|_{L^{\infty}(\Omega)} = \calO(\eps^2)\quad\text{as }\eps \searrow 0
\end{align*}
for any choice of $f,g$.
\end{conjecture}

A major contribution of this work is to provide a proof of this conjecture as well as a generalization to more general types of diffusion matrices in any dimension. 
We then give a complete characterization of type-$\eps^2$ diagonal diffusion matrices in two dimensions and a systematic study in higher dimensions.

The main challenge in characterizing type-$\eps^2$ diffusion matrices lies in the fact that, in general, the invariant measure $r$ and the correctors $v^{kl}$ depend on $A$ in a highly nontrivial way, and that the evaluation of $c^{kl}_j(A)$ involves understanding the complicated relation between the invariant measure and the gradient of the corrector.

To the best of our knowledge,  prior to this work, almost all known examples of type-$\eps^2$ diffusion matrices fall into the {\it reversible} category where $\mathrm{div}(rA) = 0$ (weakly). The term ``reversible" stems from the fact that, in this case, the operator $u\mapsto -A:D^2u$ is self-adjoint in $L^2(r)$ and describes a diffusion which is time-reversible; see e.g., \cite{DSF88}. Note that it is immediately clear from the definition \eqref{eq:def-cjkl} of $c_j^{kl}(A)$ that $\mathrm{div}(rA) = 0$ implies that $A$ is type-$\eps^2$, and no information on the corrector is needed in this case.  However, in the non-reversible regime,  the complicated relation between $r, A$ and $\nabla v^{kl}$ renders the equalities $C_{jkl}(A)=0$ almost impossible to verify.  Moreover, the definitions of $c_j^{kl}(A)$ and $C_{jkl}(A)$ tell us only very little about the structure of type-$\eps^2$ diffusion matrices. For example, it is not clear from \eqref{eq:def-cjkl} how rare type-$\eps^2$ maps are, or how to construct a type-$\eps$ map.  In view of the above discussion, it is important to find further characterizations to achieve better verifiability and obtain a clearer understanding of the structure of type-$\eps^2$ diffusion matrices.

\subsection{Main results}

As our first main result, we confirm Conjecture~\ref{Conj} by obtaining a new class of type-$\eps^2$ diffusion matrices. This shows that type-$\eps^2$ diffusion matrices, although generically very rare, are far more abundant than the reversible case.

\begin{theorem}[A new class of type-$\eps^2$ maps]\label{Thm: Main}
Let $C,M\in \R^{n\times n}_{\mathrm{sym}}$, let $a\in C^{0,\alpha}(\T^n)$ for some $\alpha \in (0,1)$, and suppose that the map $A:\R^n \rightarrow \R^{n\times n}_{\mathrm{sym}}$ given by
\begin{align*}
A(y):= C + a(y) M\quad\text{for}\quad y\in \R^n
\end{align*}
satisfies $A\in C^{0,\alpha}(\T^n;\calS^n_+)$. Then, there holds $c_j^{kl}(A)=0$ for all $j,k,l\in \{1,\dots,n\}$.
In particular, $A$ is type-$\eps^2$.
\end{theorem}

As an immediate consequence, we find that any $A\in C^{0,\alpha}(\T;(0,\infty))$ is type-$\eps^2$, i.e., any diffusion matrix in dimension $n=1$ is type-$\eps^2$. More interestingly, we can deduce from Theorem \ref{Thm: Main} that any diagonal diffusion matrix with constant trace is type-$\eps^2$ in dimension $n=2$.

\begin{corollary}[Diagonal constant-trace maps are type-$\eps^2$ in $n=2$]\label{Cor: Cor1}
Let $c>0$ and $\alpha \in (0,1)$. Then, every diagonal map $A\in C^{0,\alpha}(\T^2;\calS^2_+)$ with $\tr(A)\equiv c$ is type-$\eps^2$. In particular, Conjecture \ref{Conj} is true.
\end{corollary}

An important special case of solutions to the Dirichlet problem are the harmonic functions.
As another interesting consequence of Theorem \ref{Thm: Main}, we show that for any diagonal $A\in C^{0,\alpha}(\T^2;\calS^2_+)$,  the ``$A(\tfrac\cdot\eps)$-harmonic" functions, i.e.,  solutions to \eqref{ueps problem} with $f\equiv 0$, homogenize at rate $\calO(\eps^2)$.

For the harmonic functions, the crucial observation is that multiplication of $A\in C^{0,\alpha}(\T^n;\calS^n_+)$ by any function $\gamma\in C^{0,\alpha}(\T^n;(0,\infty))$ does not change the solution of \eqref{ueps problem} when $f\equiv 0$. Hence, harmonic functions homogenize at rate $\calO(\eps^2)$ as long as the ``orbit" $\Gamma_A:=\{\gamma A: \gamma\in C^{0,\alpha}(\T^n;(0,\infty))\}$ contains at least one type-$\eps^2$ map. By Corollary~\ref{Cor: Cor1}, this is certainly true when $n=2$, in which case $\tfrac 1{\tr(A)}A$ is type-$\eps^2$.

\begin{corollary}[Convergence rate $\calO(\eps^2)$ for diagonal maps in $n=2$ when $f\equiv 0$]\label{Cor: Cor2}
Let $a_1,a_2\in C^{0,\alpha}(\T^2;(0,\infty))$ for some $\alpha\in (0,1)$ and define $A\in C^{0,\alpha}(\T^2;\calS^2_+)$ by 
\begin{align*}
A(y) := \mathrm{diag}(a_1(y),a_2(y))\quad\text{for}\quad y\in \R^2.
\end{align*}
Set $f\equiv 0$. Then, for any choice of $g$ satisfying \eqref{f,g ass}, there holds
\begin{align*}
\left\|u^{\eps} - u \right\|_{L^{\infty}(\Omega)} = \calO(\eps^2)\quad\text{as }\eps \searrow 0,
\end{align*}
where $u^{\eps}$ denotes the solution to \eqref{ueps problem} and $u$ the solution to \eqref{u problem}.
\end{corollary}

\begin{remark}\label{Rk: Cor of Thm 1.1 not true in n>2}
Corollaries \ref{Cor: Cor1} and \ref{Cor: Cor2} are not true in dimensions $n\geq 3$:
\begin{itemize}
\item[(i)] There exists a diagonal map $A\in C^{\infty}(\T^3;\calS^3_+)$ with constant trace which is type-$\eps$; see Section \ref{Subsec: c-bad diagonal map with constant trace in $n=3$}.
\item[(ii)] 
There exists a diagonal map $A\in C^{\infty}(\T^3;\calS^3_+)$ and a function $g\in C^{\infty}(\bar{\Omega})$ such that the optimal rate of convergence of the solutions $u^{\eps}$ of \eqref{ueps problem} with $f\equiv 0$ in the $L^{\infty}(\Omega)$-norm is $\calO(\eps)$; see Section \ref{Subsec: Example with optimal rate Oeps}.
\end{itemize}
\end{remark}

The second major result of this work 
is the complete characterization of type-$\eps^2$ diagonal maps in dimension $n=2$.

\begin{theorem}[type-$\eps^2$ diagonal maps in $n=2$]\label{Thm: Complete char}
Let $a\in C^{0,\alpha}(\T^2;(0,\infty))$ and $b\in C^{0,\alpha}(\T^2;(-1,1))$ for some $\alpha \in (0,1)$. Further, let $A,B\in C^{0,\alpha}(\T^2;\calS^2_+)$ be the maps defined by
\begin{align}\label{form for char}
A(y):=a(y)B(y),\qquad   B(y) := \mathrm{diag}(1+b(y),1-b(y))
\end{align}
for $y\in \R^2$. We denote the invariant measure of $A$ by $r\in C^{0,\alpha}(\T^2;(0,\infty))$ and the invariant measure of $B$ by $r_B\in C^{0,\alpha}(\T^2;(0,\infty))$. We introduce the functions $w_A,w_B \in C^{2,\alpha}(\T^n)$ as the unique solutions to
\begin{align*}
-A:D^2 w_A = a-\int_Y ra&\quad\text{in }Y,\qquad w_A \text{ is }Y\text{-periodic},\qquad \int_Y w_A = 0,\\
-\Delta w_B = r_B\, b-\int_Y r_B\, b&\quad\text{in }Y,\qquad w_B \text{ is }Y\text{-periodic},\qquad \int_Y w_B = 0.
\end{align*}
Then, $A$ is type-$\eps^2$ if and only if
\begin{align}\label{eq:wawb}
\int_Y (\partial_1 w_A)(\partial_{22}^2 w_B) = \int_Y (\partial_2 w_A)(\partial_{11}^2 w_B) = 0.
\end{align}
\end{theorem}

Theorem \ref{Thm: Complete char} is a characterization of all type-$\eps^2$ diagonal diffusion matrices in dimension $n=2$ since any diagonal $A\in C^{0,\alpha}(\T^2;\calS^2_+)$ can be written as \eqref{form for char} with $a:= \frac{1}{2}\tr (A)$. In particular, when $A$ has constant trace we have $w_A\equiv 0$ and thus, $A$ is type-$\eps^2$, which is consistent with the result of Corollary \ref{Cor: Cor1}. 
On the other hand, to construct a type-$\eps$ map, one may choose $a$ such that \eqref{eq:wawb} fails.  

Another feature of this characterization is that, for fixed $B$, one can design $a$ explicitly such that,  up to a multiplicative constant, $w_A$ can be any $Y$-periodic function with mean zero. Indeed, given a function $\fhi \in C^{2,\alpha}(\T^2)$ with $\int_Y \fhi = 0$, if we set $a:=\frac{1}{1+sB:D^2\fhi}$ with $s>0$ such that $a\in C^{0,\alpha}(\T^2;(0,\infty))$, then $w_A = s\fhi$.
\medskip

As mentioned above,  the existence of a type-$\eps^2$ ``representative" within the orbit $\Gamma_A$ of $A\in C^{0,\alpha}(\T^n;\calS^n_+)$ is crucial for the quantification of the homogenization of harmonic functions. Our third main result reveals the representatives of the orbit of any diagonal diffusion matrix. To be specific, it states that if the orbit of a diagonal diffusion matrix $A$ contains a type-$\eps^2$ diffusion matrix $\tilde{A}\in C^{0,\alpha}(\T^n;\calS^n_+)$, then $\frac{1}{\gamma}\tilde{A}$ is type-$\eps^2$, where $\gamma\in C^{0,\alpha}(\T^n;(0,\infty))$ is any positive linear combination of entries of $\tilde{A}$.  In particular,  diagonal coefficient matrices with constant trace can serve as such representatives.

\begin{theorem}[Classification of $\frac{1}{C:A}A$]\label{Thm: 1/C:A A}
Let $A\in C^{0,\alpha}(\T^n;\calS^n_+)$ for some $\alpha \in (0,1)$. Let $C\in \R^{n\times n}$ be a constant matrix such that the function $\gamma:\R^n\rightarrow \R$ given by
\begin{align*}
\gamma(y):=\frac{1}{C:A(y)}\quad\text{for}\quad y\in \R^n
\end{align*}
satisfies $\gamma \in C^{0,\alpha}(\T^n;(0,\infty))$. Introducing $\tilde{A}:=\gamma A$, the following claims hold.
\begin{itemize}
\item[(i)] If $c_j^{kl}(A)=0$ for all $j,k,l\in \{1,\dots,n\}$, then also $c_j^{kl}(\tilde A)=0$ for all $j,k,l\in \{1,\dots,n\}$.

\item[(ii)] If $A$ is diagonal and type-$\eps^2$, then also $\tilde A$ is type-$\eps^2$.
\end{itemize}
\end{theorem}

As the final major result of this paper, we show that the set of type-$\eps$ maps is open and dense in $C^{0,\alpha}(\T^n;\calS^n_+)$.

\begin{theorem}[Density of type-$\eps$ maps in $C^{0,\alpha}(\T^n;\calS^n_+)$]\label{Thm: density}
Let $n\geq 2$ and $\alpha\in (0,1)$. Then, the subset 
\begin{align*}
\{A\in C^{0,\alpha}(\T^n;\calS^n_+):A\text{ is type-$\eps$}\}
\end{align*}
is open and dense in $C^{0,\alpha}(\T^n;\calS^n_+)$.
\end{theorem}

We provide a further systematic study of type-$\eps^2$ and type-$\eps$ maps in Section \ref{sec:further}. See Lemmata \ref{Thm: explicit perturbation}--\ref{Thm: scalar multiples} and Remarks \ref{Rk: Example recover ST}--\ref{Rk: Suff cond for c-good} for the results.

\subsection{Structure of the paper}

In Section \ref{Sec: Pfs of main results}, we present the proofs of the main results, i.e., we prove Theorems \ref{Thm: Main}--\ref{Thm: density} and Corollaries \ref{Cor: Cor1}--\ref{Cor: Cor2}. 

In Section \ref{sec:further}, we conduct a further analysis on type-$\eps^2$ and type-$\eps$ diffusion matrices. More precisely, we provide perturbation arguments for the creation of type-$\eps$ maps (Section \ref{Subsec: Pertarg}), we provide a criterion for the classification of diffusion matrices $A$ for which $r,rA$ have a special structure (Section \ref{Subsec: 3.2}), we investigate the existence of type-$\eps$ diffusion matrices with constant trace in two dimensions (Section \ref{Subsec: 3.new}), we study scalar multiples and sums of type-$\eps^2$ maps (Section \ref{Subsec: 3.3}), we study the case $r\equiv 1$ (Section \ref{Subsec: 3.4}), and we provide some sufficient conditions for a map to be type-$\eps^2$ (Section \ref{Subsec: 3.5}).

In Section \ref{Sec: 4}, we collect various counterexamples to demonstrate the results of Remark \ref{Rk: Cor of Thm 1.1 not true in n>2} (Sections \ref{Subsec: c-bad diagonal map with constant trace in $n=3$}--\ref{Subsec: Example with optimal rate Oeps}) and Remarks \ref{Rk: A+I}--\ref{Rk: Mult by r} (Sections \ref{Subsec: A+I}--\ref{Subsec: c-bad diagonal r=1 n=2}).

Finally, in Section \ref{Sec: Conc}, we give some concluding remarks as well as open problems.

\section{Proofs of the main results}\label{Sec: Pfs of main results}

In this section, we prove the main results of this paper, namely Theorems \ref{Thm: Main}--\ref{Thm: density} and Corollaries \ref{Cor: Cor1}--\ref{Cor: Cor2}. 

\subsection{Proof of Theorem \ref{Thm: Main} and its corollaries}\label{Subsec: Pf of Thm Main and cor}

\begin{proof}[Proof of Theorem \ref{Thm: Main}]
Let $C,M\in \R^{n\times n}_{\mathrm{sym}}$, let $a\in C^{0,\alpha}(\T^n)$ for some $\alpha \in (0,1)$, and suppose that the map $A:\R^n \rightarrow \R^{n\times n}_{\mathrm{sym}}$ given by
\begin{align}\label{Pf Thm 1.1: A=C+aM}
A(y):= C + a(y) M \quad\text{for}\quad y\in \R^n
\end{align}
satisfies $A\in C^{0,\alpha}(\T^n;\calS^n_+)$. Let $r\in C^{0,\alpha}(\T^n;(0,\infty))$ denote the invariant measure of $A$ defined as the solution to \eqref{r problem}, let 
\begin{align}\label{Pf Thm 1.1: abar}
\bar{a}:=\int_Y ra,
\end{align}
and let $v^{kl}\in C^{2,\alpha}(\T^n)$ denote the solution to the (k,l)-th cell problem \eqref{vij problem} for $1\leq k,l\leq n$. 
We show that
\begin{align*}
c_j^{kl}(A) = \int_Y rAe_j\cdot \nabla v^{kl} = 0\qquad \forall\, 1\leq j,k,l\leq n.
\end{align*}
\textbf{Step 0:} We introduce the function $w\in C^{2,\alpha}(\T^n)$ to be the unique solution to
\begin{align}\label{Pf Thm 1.1: w}
-A:D^2 w = a-\bar{a}\quad\text{in }Y,\qquad w \text{ is }Y\text{-periodic},\qquad \int_Y w = 0.
\end{align}
Note that $w$ is well-defined since $\int_Y r(a-\bar{a}) = 0$. Further, we define 
\begin{align}\label{Pf Thm 1.1: xi eta}
\xi := 1 + M:D^2 w \in C^{0,\alpha}(\T^n),\qquad \eta^{kl}:=m_{kl} w\in C^{2,\alpha}(\T^n)
\end{align}
for $1\leq k,l\leq n$.

\medskip
\noindent \textbf{Step 1:} We prove that $r = \xi$. Note that $\xi\in C^{0,\alpha}(\T^n)$ and $\int_Y \xi = 1$. For any $\fhi\in C^{\infty}(\T^n)$ we have that 
\begin{align*}
\int_Y \xi (-A:D^2 \fhi) &= \int_Y (1-\xi) (C:D^2 \fhi) + \int_Y ( \bar{a} - \xi a) (M:D^2 \fhi) \\ &=\int_Y (-C:D^2 w) (M:D^2 \fhi) - \int_Y ( aM:D^2 w + a-\bar{a}) (M:D^2 \fhi) \\&= \int_Y (-A:D^2 w -(a - \bar{a}))(M:D^2 \fhi) \\ &= 0,
\end{align*}
where we have used in the first equality \eqref{Pf Thm 1.1: A=C+aM} and the fact that $\int_Y D^2 \fhi = 0$, in the second equality the definition of $\xi$ from \eqref{Pf Thm 1.1: xi eta} and the fact that by integration by parts there holds
\begin{align*}
\int_Y (M:D^2 w) (C:D^2 \fhi) &= \sum_{i,j,k,l=1}^n m_{ij}c_{kl} \int_Y \partial_{ij}^2 w\,  \partial_{kl}^2 \fhi \\&= \sum_{i,j,k,l=1}^n m_{ij}c_{kl} \int_Y \partial_{kl}^2 w\,  \partial_{ij}^2 \fhi = \int_Y (C:D^2 w) (M:D^2 \fhi),
\end{align*}
in the third equality \eqref{Pf Thm 1.1: A=C+aM}, and in the fourth equality that $w$ satisfies \eqref{Pf Thm 1.1: w}. It follows that $\xi\in C^{0,\alpha}(\T^n)$ is a solution to
\begin{align*}
-D^2:(\xi A) = 0\quad\text{in }Y,\qquad \xi \text{ is }Y\text{-periodic},\qquad \int_Y \xi = 1,
\end{align*}
and hence, by the uniqueness of solutions, we find that 
\begin{align}\label{Pf Thm 1.1: r=xi}
r = \xi = 1 + M:D^2 w.
\end{align}
\textbf{Step 2:} We prove that $v^{kl} = \eta^{kl}$ for all $1\leq k,l\leq n$. Note that we have $\eta^{kl}\in C^{2,\alpha}(\T^n)$ and $\int_Y \eta^{kl} = 0$. Further, by the definition of $\eta^{kl}$ from \eqref{Pf Thm 1.1: xi eta}, the fact that $w$ satisfies \eqref{Pf Thm 1.1: w}, and the definition of $\bar{a}$ from \eqref{Pf Thm 1.1: abar}, we have that
\begin{align*}
-A:D^2 \eta^{kl} = m_{kl}(a-\bar{a}) = (c_{kl}+m_{kl} a) - \int_Y r(c_{kl}+m_{kl} a) = a_{kl} - \int_Y r a_{kl}.
\end{align*}
It follows that $\eta^{kl}\in C^{2,\alpha}(\T^n)$ is a solution to
\begin{align*}
-A:D^2 \eta^{kl} = a_{kl}-\int_Y r a_{kl}\quad\text{in }Y,\qquad \eta^{kl} \text{ is }Y\text{-periodic},\qquad \int_Y \eta^{kl} = 0,
\end{align*} 
and hence, by uniqueness of solutions, we find that
\begin{align}\label{Pf of Thm 1.1: vkl=etakl}
v^{kl} = \eta^{kl} = m_{kl} w\qquad \forall\, 1\leq k,l\leq n.
\end{align}
\textbf{Step 3:} We show that $c_j^{kl}(A)=0$ for all $1\leq j,k,l\leq n$. First, we note that by the definition of $A$ from \eqref{Pf Thm 1.1: A=C+aM}, the fact that $w$ satisfies \eqref{Pf Thm 1.1: w}, and the relation \eqref{Pf Thm 1.1: r=xi} from Step 1, there holds
\begin{align}\label{Pf of Thm 1.1: -C:D^2w}
-C:D^2 w = -A:D^2 w + aM:D^2 w = a-\bar{a} + (r-1)a = ra - \bar{a}.
\end{align}
For any $1\leq j,k,l\leq n$, we find that
\begin{align*}
c_j^{kl}(A) &= \int_Y rAe_j\cdot \nabla v^{kl} \\&= \int_Y (r-1)(Ce_j\cdot \nabla v^{kl})  + \int_Y (ra-\bar{a})(Me_j\cdot \nabla v^{kl}) \\ &=m_{kl}\int_Y (M:D^2 w)(Ce_j\cdot \nabla w)  - m_{kl}\int_Y (C:D^2 w)(Me_j\cdot \nabla w) \\&= 0,
\end{align*}
where we have used in the second equality that $\int_Y \nabla v^{kl} = 0$, in the third equality the relations \eqref{Pf Thm 1.1: r=xi}, \eqref{Pf of Thm 1.1: vkl=etakl} from Step 2 and \eqref{Pf of Thm 1.1: -C:D^2w}, and in the fourth equality that
\begin{align*}
\int_Y (\partial_{st}^2 w) (\partial_i w) = 0\qquad \forall 1\leq i,s,t\leq n
\end{align*}
which holds as integration by parts shows that $\int_Y (\partial_{st}^2 w) (\partial_i w) = -\int_Y (\partial_i w)(\partial_{st}^2 w)$.
\end{proof}

Let us note that the function $w$ from \eqref{Pf Thm 1.1: w} is the key to the proof of Theorem \ref{Thm: Main}. Indeed, $w$ encompasses the information needed about $a$ to obtain explicit formulas for the invariant measure $r$ (see \eqref{Pf Thm 1.1: r=xi}) and the correctors $v^{kl}$ (see \eqref{Pf of Thm 1.1: vkl=etakl}). 

We are now in a position to give quick proofs of Corollaries \ref{Cor: Cor1}--\ref{Cor: Cor2}.

\begin{proof}[Proof of Corollary \ref{Cor: Cor1}]
Let $c> 0$ be a positive constant, $a\in C^{0,\alpha}(\T^2;(0,c))$ for some $\alpha \in (0,1)$, and let $A\in C^{0,\alpha}(\T^2;\calS^2_+)$ be the map given by
\begin{align*}
A(y):= \mathrm{diag}(a(y),c-a(y)) \quad\text{for}\quad y\in \R^2.
\end{align*}
Then, setting $C:=\mathrm{diag}(0,c)\in \R^{2\times 2}_{\mathrm{sym}}$ and $M:=\mathrm{diag}(1,-1)\in \R^{2\times 2}_{\mathrm{sym}}$, we have that
\begin{align*}
A(y) = C + a(y) M 
\end{align*}
for any $y\in \R^2$, and thus, $A$ is type-$\eps^2$ by Theorem \ref{Thm: Main}.
\end{proof}

\begin{proof}[Proof of Corollary \ref{Cor: Cor2}]
Let $A\in C^{0,\alpha}(\T^2;\calS^2_+)$ with $\alpha \in (0,1)$ be of the form 
\begin{align*}
A(y) = \mathrm{diag}(a_1(y),a_2(y))\quad\text{for}\quad y\in \R^2
\end{align*}
for some $a_1,a_2\in C^{0,\alpha}(\T^2;(0,\infty))$, and let $g\in C(\bar{\Omega})$ be such that $g\in W^{5,q}(\Omega)$ for some $q>n$. We need to show that the solution $u^{\eps}$ of 
\begin{align*}
\left\{\begin{aligned}-A\left(\frac{\cdot}{\eps}\right):D^2 u^{\eps} &= 0& &\text{in }\Omega,\\
u^{\eps} &= g&  &\text{on }\partial\Omega,\end{aligned}\right.
\end{align*}
converges to the solution $u$ of the homogenized problem in the $L^{\infty}(\Omega)$-norm with rate of convergence 
\begin{align}\label{Pf of Cor 1.2: to show}
\left\|u^{\eps} - u \right\|_{L^{\infty}(\Omega)} = \calO(\eps^2)\quad\text{as }\eps \searrow 0.
\end{align}
To this end, note that $u^{\eps}$ is also the unique solution to the problem
\begin{align*}
\left\{\begin{aligned}-\tilde{A}\left(\frac{\cdot}{\eps}\right):D^2 u^{\eps} &= 0& &\text{in }\Omega,\\
u^{\eps} &= g&  &\text{on }\partial\Omega,\end{aligned}\right.
\end{align*}
with $\tilde{A}\in C^{0,\alpha}(\T^2;\calS^2_+)$ defined as $\tilde{A}(y) := \frac{1}{a_1(y) + a_2(y)}A(y)$ for $y\in \R^2$. Noting that $\tilde{A}$ is of the form $\tilde{A} = \mathrm{diag}(\tilde{a}_1,\tilde{a}_2)$ for some $\tilde{a}_1,\tilde{a}_2\in C^{0,\alpha}(\T^2;(0,\infty))$ with $\tilde{a}_1+\tilde{a}_2\equiv 1$, we deduce from Corollary \ref{Cor: Cor1} that $\tilde{A}$ is type-$\eps^2$, and \eqref{Pf of Cor 1.2: to show} follows.
\end{proof}

\subsection{Proof of Theorem \ref{Thm: Complete char}}\label{Subsec: Pf of Thm 1.2}

We start by proving a lemma which will be helpful throughout this work.

\begin{lemma}\label{Lmm: aB}
Let $a\in C^{0,\alpha}(\T^n;(0,\infty))$, $B\in C^{0,\alpha}(\T^n;\calS^n_+)$ for some $\alpha \in (0,1)$, and consider the map $A\in C^{0,\alpha}(\T^n;\calS^n_+)$ given by 
\begin{align*}
A(y):=a(y) B(y)\quad\text{for}\quad y\in \R^n.
\end{align*}
Let $r_B\in C^{0,\alpha}(\T^n;(0,\infty))$, $\bar{B}:=\int_Y r_B B\in \calS^n_+$, $v^{kl}_B\in C^{2,\alpha}(\T^n)$ denote the invariant measure, effective coefficient and correctors corresponding to $B$, respectively. 
Further, denoting the invariant measure of $A$ by $r\in C^{0,\alpha}(\T^n;(0,\infty))$, we write $\bar{a}:=\int_Y ra > 0$ and introduce the function $w\in C^{2,\alpha}(\T^n)$ to be the unique solution to
\begin{align}\label{w for lemma}
-A:D^2 w = a-\bar{a}\quad\text{in }Y,\qquad w \text{ is }Y\text{-periodic},\qquad \int_Y w = 0.
\end{align}
Then, the following assertions hold.
\begin{itemize}
\item[(i)] The invariant measure $r$ of $A$ is given by
\begin{align*}
r =d\,\frac{r_B}{a}\quad\text{ where }\quad d:= \left(\int_Y \frac{r_B}{a}\right)^{-1} = \bar{a},
\end{align*}
and the effective coefficient corresponding to $A$ is given by $\bar{A} = \bar{a}\bar{B}$.

\item[(ii)] The correctors $v^{kl}$ of $A$ are given by
\begin{align*}
v^{kl} = v^{kl}_B + \bar{b}_{kl} w,\qquad k,l\in\{1,\dots,n\}.
\end{align*}
\item[(iii)] The values $c_j^{kl}(A)$ are given by
\begin{align*}
c_j^{kl}(A) = \bar{a}\left(c_j^{kl}(B)+\bar{b}_{kl} \int_Y r_B B e_j\cdot \nabla w  \right),\qquad j,k,l\in\{1,\dots,n\}.
\end{align*}
\end{itemize}
\end{lemma}

\begin{proof}
(i) It is quickly checked that the function $r:= d \frac{r_B}{a}$ with $d:= \left(\int_Y \frac{r_B}{a}\right)^{-1}$ satisfies $r\in C^{0,\alpha}(\T^n;(0,\infty))$, $\int_Y r = 1$, and we have that
\begin{align*}
\int_Y r(-A:D^2 \fhi) = \int_Y \left(d\,\frac{r_B}{a} \right)(-aB:D^2 \fhi) = d\int_Y r_B(-B:D^2 \fhi) = 0
\end{align*}
for all $\fhi\in C^{\infty}(\T^n)$. Hence, $r$ is indeed the invariant measure of $A$. We conclude the proof of (i) by noting that $d = \int_Y d\,\frac{r_B}{a}a = \int_Y ra = \bar{a}$ and
\begin{align*}
\bar{A} = \int_Y rA = \int_Y \left(\bar{a}\frac{r_B}{a}\right)aB = \bar{a}\int_Y r_B B = \bar{a}\bar{B}.
\end{align*}
(ii) We observe that the function $v^{kl}:=v^{kl}_B+ \bar{b}_{kl} w$ satisfies $v^{kl}\in C^{2,\alpha}(\T^n)$, $\int_Y v^{kl} = 0$ and we have that
\begin{align*}
-A:D^2 v^{kl} &= a\left(-B:D^2 v^{kl}_B\right) +\bar{b}_{kl} \left(- A:D^2 w\right) \\ &= a(b_{kl} - \bar{b}_{kl}) + \bar{b}_{kl}(a-\bar{a}) \\ &= a_{kl} - \bar{a}_{kl}
\end{align*}
for any $1\leq k,l\leq n$, where we have used that $A=aB$, the fact that $v^{kl}_B$ solves the (k,l)-th cell problem corresponding to $B$ and $w$ solves \eqref{w for lemma}, and the result from (i) that $\bar{A}=\bar{a}\bar{B}$.

\medskip
\noindent(iii) We use the definition \eqref{eq:def-cjkl} of $c_j^{kl}(A)$ and $c_j^{kl}(B)$, as well as the results from (i) and (ii) to obtain that
\begin{align*}
\frac{c_j^{kl}(A)}{\bar{a}} = \int_Y \frac{r}{\bar{a}}A e_j\cdot \nabla v^{kl} = \int_Y r_B B e_j \cdot \nabla\left(v^{kl}_B+ \bar{b}_{kl} w \right) = c_j^{kl}(B)+\bar{b}_{kl} \int_Y r_B B e_j\cdot \nabla w  
\end{align*} 
for any $1\leq j,k,l\leq n$.
\end{proof}

\begin{proof}[Proof of Theorem \ref{Thm: Complete char}]
Let $a\in C^{0,\alpha}(\T^2;(0,\infty))$ and $b\in C^{0,\alpha}(\T^2;(-1,1))$ for some $\alpha \in (0,1)$. Further, let $B\in C^{0,\alpha}(\T^2;\calS^2_+)$ be the map defined as
\begin{align}\label{Pf of compl char: B=diag}
B(y) := \mathrm{diag}(1+b(y),1-b(y))\quad\text{for}\quad y\in \R^2,
\end{align}
and let $A\in C^{0,\alpha}(\T^2;\calS^2_+)$ be the map defined as
\begin{align*}
A(y):= a(y)B(y)\quad\text{for}\quad y\in \R^2.
\end{align*}
We denote the invariant measure of $A$ by $r\in C^{0,\alpha}(\T^2;(0,\infty))$ (see \eqref{r problem}), the effective coefficient to $A$ by $\bar{A}:=\int_Y rA \in \calS^2_+$, and the solution to the (k,l)-th cell problem \eqref{vij problem} corresponding to $A$ by $v^{kl}\in C^{2,\alpha}(\T^2)$ for $k,l\in \{1,2\}$. We denote the invariant measure of $B$ by $r_B\in C^{0,\alpha}(\T^2;(0,\infty))$, the effective coefficient to $B$ by $\bar{B}:=\int_Y r_B B \in \calS^2_+$, and we denote the solution to the (k,l)-th cell problem corresponding to $B$ by $v^{kl}_B\in C^{2,\alpha}(\T^2)$ for $k,l\in \{1,2\}$. Further, we write
\begin{align*}
\bar{a}:= \int_Y ra,\qquad \bar{b}:=\int_Y r_B b
\end{align*}
and we note that $\bar{a}>0$ and $\bar{b} \in (-1,1)$, where the latter follows from the fact that 
\begin{align}\label{Pf of compl char: Bbar}
\bar{B} = \mathrm{diag}(1+\bar{b},1-\bar{b})\in \calS^2_+.
\end{align}
\textbf{Step 0:} We introduce the function $w_A\in C^{2,\alpha}(\T^2)$ to be the unique solution to
\begin{align*}
-A:D^2 w_A = a-\bar{a}\quad\text{in }Y,\qquad w_A \text{ is }Y\text{-periodic},\qquad \int_Y w_A = 0,
\end{align*}
and the function $w_B\in C^{2,\alpha}(\T^2)$ to be the unique solution to
\begin{align*}
-B:D^2 w_B = b-\bar{b}\quad\text{in }Y,\qquad w_B \text{ is }Y\text{-periodic},\qquad \int_Y w_B = 0.
\end{align*}
Note that $w_A$ and $w_B$ are well-defined since $\int_Y r(a-\bar{a}) = 0$ and $\int_Y r_B(b-\bar{b}) = 0$.

\medskip
\noindent \textbf{Step 1:} Since $B$ is of the form 
\begin{align*}
B = C + bM\qquad\text{with}\qquad C:=I_2,\quad M:=\mathrm{diag}(1,-1), 
\end{align*}
we know from the proof of Theorem \ref{Thm: Main} that $r_B = 1 + M:D^2 w_B$, i.e., 
\begin{align}\label{Pf of Thm compl char: rB}
r_B = 1 + \partial_{11}^2 w_B -\partial_{22}^2 w_B;
\end{align}
see \eqref{Pf Thm 1.1: r=xi}, that $v^{kl}_B = m_{kl} w_B$ for $k,l\in\{1,2\}$, i.e.,
\begin{align*}
v^{11}_B  = w_B,\qquad v^{22}_B = -w_B,\qquad v^{12}_B = v^{21}_B \equiv 0;
\end{align*}
see \eqref{Pf of Thm 1.1: vkl=etakl}, and that $w_B$ satisfies $-C:D^2 w_B = r_B b - \bar{b}$, i.e., $w_B$ can be equivalently characterized as the unique solution to 
\begin{align}\label{Pf of Thm compl char: Lapw}
-\Delta w_B = r_B b-\bar{b}\quad\text{in }Y,\qquad w_B \text{ is }Y\text{-periodic},\qquad \int_Y w_B = 0;
\end{align}
see \eqref{Pf of Thm 1.1: -C:D^2w}.

\medskip
\noindent \textbf{Step 2:} We apply Lemma \ref{Lmm: aB}(iii) to find that for any $j,k,l\in \{1,2\}$ we have
\begin{align*}
c_j^{kl}(A) = \bar{a}\left(c_j^{kl}(B)+\bar{b}_{kl} \int_Y r_B B e_j\cdot \nabla w_A  \right) = \bar{a}\,\bar{b}_{kl} \int_Y r_B B e_j\cdot \nabla w_A,  
\end{align*}
where we have used in the second equality that there holds $c_j^{kl}(B) = 0$ for all $j,k,l\in \{1,2\}$ by Theorem \ref{Thm: Main}. Hence, using \eqref{Pf of compl char: Bbar}, we find that
\begin{align*}
c_j^{11}(A) = \bar{a}(1+\bar{b}) \int_Y r_B Be_j\cdot \nabla w_A,\qquad c_j^{22}(A) =\bar{a}(1-\bar{b}) \int_Y r_B Be_j\cdot \nabla w_A
\end{align*}
for any $j\in \{1,2\}$, and $c_j^{kl}(A) = 0$ whenever $k\neq l$. We can simplify further by noting that
\begin{align*}
r_B b_{11} &= r_B + (r_B b -\bar{b}) + \bar{b}   =  1 + \partial_{11}^2 w_B -\partial_{22}^2 w_B  - \Delta w_B +\bar{b} = (1+\bar{b}) -2\,\partial_{22}^2 w_B,\\
r_B b_{22} &= r_B - (r_B b -\bar{b}) -\bar{b}  =  1 + \partial_{11}^2 w_B -\partial_{22}^2 w_B  + \Delta w_B -\bar{b}= (1-\bar{b}) +2\,\partial_{11}^2 w_B,
\end{align*}
where we have used the definition of $B$ from \eqref{Pf of compl char: B=diag}, the formula for $r_B$ from \eqref{Pf of Thm compl char: rB}, and the fact that $w_B$ satisfies \eqref{Pf of Thm compl char: Lapw}. Therefore, the values $c^{11}_j(A)$ for $j\in\{1,2\}$ are given by
\begin{align*}
c_1^{11}(A) &= \bar{a}(1+\bar{b}) \int_Y r_B b_{11}\partial_1 w_A = -2\bar{a}(1+\bar{b}) \int_Y (\partial_1 w_A)(\partial_{22}^2 w_B),\\
c_2^{11}(A) &= \bar{a}(1+\bar{b}) \int_Y r_B b_{22}\partial_2 w_A = 2\bar{a}(1+\bar{b}) \int_Y (\partial_2 w_A)(\partial_{11}^2 w_B),
\end{align*}
and the values $c^{22}_j(A)$ for $j\in\{1,2\}$ are given by
\begin{align*}
c_1^{22}(A) &= \bar{a}(1-\bar{b}) \int_Y r_B b_{11}\partial_1 w_A = -2\bar{a}(1-\bar{b}) \int_Y (\partial_1 w_A)(\partial_{22}^2 w_B),\\
c_2^{22}(A) &= \bar{a}(1-\bar{b}) \int_Y r_B b_{22}\partial_2 w_A = 2\bar{a}(1-\bar{b}) \int_Y (\partial_2 w_A)(\partial_{11}^2 w_B).
\end{align*} 
We conclude that $A$ is type-$\eps^2$ if and only if
\begin{align*}
\int_Y (\partial_1 w_A)(\partial_{22}^2 w_B) = \int_Y (\partial_2 w_A)(\partial_{11}^2 w_B) = 0.
\end{align*}
The proof is complete.
\end{proof}

\subsection{Proof of Theorem \ref{Thm: 1/C:A A}}\label{Subsec: Pf of Thm 1.3}

\begin{proof}[Proof of Theorem \ref{Thm: 1/C:A A}]

Let $A\in C^{0,\alpha}(\T^n;\calS^n_+)$ with $\alpha\in (0,1)$ and let $C\in \R^{n\times n}$ be a constant matrix such that for the function $\gamma:\R^n\rightarrow \R$ defined by
\begin{align}\label{Pf of Thm 1/C:A: gamma}
\gamma(y):=\frac{1}{C:A(y)}\quad\text{for}\quad y\in \R^n
\end{align}
there holds $\gamma\in C^{0,\alpha}(\T^n;(0,\infty))$. We set 
\begin{align*}
\tilde{A}:=\gamma A \in C^{0,\alpha}(\T^n;\calS^n_+).
\end{align*}
We only need to prove assertion (i), i.e., that if $c_j^{kl}(A) = 0$ for all $1\leq j,k,l\leq n$, then also $c_j^{kl}(\tilde{A}) = 0$ for all $1\leq j,k,l\leq n$. We denote the invariant measure of $A$ by $r\in C^{0,\alpha}(\T^n;(0,\infty))$ (see \eqref{r problem}), the effective coefficient to $A$ by $\bar{A}:=\int_Y rA \in \calS^n_+$, and the solution to the (k,l)-th cell problem \eqref{vij problem} corresponding to $A$ by $v^{kl}$ for $1\leq k,l\leq n$. 

\medskip
\noindent \textbf{Step 1:} By Lemma \ref{Lmm: aB}, we have for any $1\leq j,k,l\leq n$ that
\begin{align}\label{Pf of Thm 1/C:A: Step 1 result}
c_j^{kl}(\tilde{A}) = \bar{\gamma}\left( c_j^{kl}(A) + \bar{a}_{kl} \int_Y r A e_j\cdot \nabla w\right),
\end{align}
where $w\in C^{2,\alpha}(\T^n)$ denotes the unique solution to
\begin{align}\label{Pf of Thm 1/C:A: fctw}
-\tilde{A}:D^2 w = \gamma-\bar{\gamma}\quad\text{in }Y,\qquad w \text{ is }Y\text{-periodic},\qquad \int_Y w = 0,
\end{align}
and $\bar{\gamma}$ denotes the positive constant 
\begin{align*}
\bar{\gamma}:= \left(\int_Y \frac{r}{\gamma}\right)^{-1} = \frac{1}{C:\bar{A}} > 0,
\end{align*}
where we have used the definition \eqref{Pf of Thm 1/C:A: gamma} of $\gamma$ and $\bar{A}=\int_Y rA$ in the second equality.

\medskip 
\noindent \textbf{Step 2:} We claim that the solution $w$ to \eqref{Pf of Thm 1/C:A: fctw} is given by
\begin{align}\label{Pf of Thm 1/C:A: Step 2 result}
w = -\bar{\gamma} \sum_{i,j=1}^n c_{ij} v^{ij}.
\end{align}
Indeed, it is quickly checked that $w\in C^{2,\alpha}(\T^n)$, $\int_Y w = 0$, and we have that
\begin{align*}
-\tilde{A}:D^2 w = -\bar{\gamma}\, \gamma \sum_{i,j=1}^n c_{ij} (-A:D^2 v^{ij})  = -\bar{\gamma}\, \gamma (C:A-C:\bar{A}) = \gamma - \bar{\gamma},
\end{align*}
where we have used that $\tilde{A} = \gamma A$, the fact that $v^{ij}$ is a solution to the (i,j)-th cell problem corresponding to $A$, and the identities $\gamma(C:A)\equiv 1$ and $\bar{\gamma}(C:\bar{A}) = 1$.

\medskip 
\noindent \textbf{Step 3:} In view of the results \eqref{Pf of Thm 1/C:A: Step 1 result} and \eqref{Pf of Thm 1/C:A: Step 2 result} from the previous steps,  we obtain that for any $1\leq j,k,l\leq n$ there holds
\begin{align*}
c_j^{kl}(\tilde{A})  &=\bar{\gamma}\left(c_j^{kl}(A) -\bar{\gamma}\,\bar{a}_{kl} \sum_{s,t=1}^n c_{st} \int_Y r A e_j\cdot \nabla v^{st}\right) \\&=\bar{\gamma}\left(c_j^{kl}(A) -\bar{\gamma}\,\bar{a}_{kl} \sum_{s,t=1}^n c_{st} \,c_j^{st}(A) \right)\\&=0,
\end{align*}
provided that $c_j^{kl}({A})=0$ for all $1\leq j,k,l\leq n$.
\end{proof}

\subsection{Proof of Theorem \ref{Thm: density}}\label{Subsec: Pf of Thm 1.4}

Let us separate the proof of Theorem \ref{Thm: density} into two parts: the first one being the proof of openness and the second one the proof of denseness of $\{A\in C^{0,\alpha}(\T^n;\calS^n_+):A\text{ is type-$\eps$}\}$ in $C^{0,\alpha}(\T^n;\calS^n_+)$ when $n\geq 2$. 

\begin{proof}[Proof of openness in Theorem \ref{Thm: density}]

Let $n\geq 2$ and $\alpha \in (0,1)$. We need to show that the set 
\begin{align*}
\{A\in C^{0,\alpha}(\T^n;\calS^n_+):A\text{ is type-$\eps$}\}
\end{align*}
is open in $C^{0,\alpha}(\T^n;\calS^n_+)$. To this end, we show that the map 
\begin{align*}
c_j^{kl}:C^{0,\alpha}(\T^n;\calS^n_+) \rightarrow \R,\qquad c_j^{kl}(A) = \int_Y r_A A e_j \cdot \nabla v^{kl}_A
\end{align*}
is continuous for any $1\leq j,k,l\leq n$, where $r_A$ denotes the invariant measure of $A$ and $v^{kl}_A$ the solution to the (k,l)-th cell problem corresponding to $A$. 

Let $(A_m)_{m\in\N}\subset C^{0,\alpha}(\T^n;\calS^n_+)$ be a sequence in $C^{0,\alpha}(\T^n;\calS^n_+)$ such that 
\begin{align*}
\|A_m - A\|_{C^{0,\alpha}(\R^n)}\longrightarrow 0\quad\text{as}\quad m\rightarrow \infty
\end{align*}
for some $A\in C^{0,\alpha}(\T^n;\calS^n_+)$. By a uniform $C^{0,\alpha}$ a priori estimate for the problem of the invariant measure (see \cite{DK17}), and writing $r_m:=r_{A_m}$ to denote the invariant measure of $A_m$, we have that
\begin{align*}
\|r_m\|_{C^{0,\alpha}(\R^n)}\leq C \|A_m\|_{C^{0,\alpha}(\R^n)}
\end{align*}
for some constant $C>0$, uniformly in $m\in \N$. It follows that the sequence $(r_m)_{m\in\N}$ is uniformly bounded in $C^{0,\alpha}(\R^n)$ and it is readily seen that $r_m\rightarrow r$ in $L^2(Y)$, where $r$ denotes the invariant measure of $A$. By a uniform $C^{2,\alpha}$ a priori estimate for the (k,l)-th cell problem from elliptic regularity theory (see \cite{GT01}), and writing $v^{kl}_m:=v^{kl}_{A_m}$, we have that
\begin{align*}
\|v^{kl}_m - v^{kl}_m(0)\|_{C^{2,\alpha}(\R^n)}\leq C(1+\|A_m\|_{C^{0,\alpha}(\R^n)})
\end{align*}
for some constant $C>0$, uniformly in $m\in \N$. It follows that $(v^{kl}_m - v^{kl}_m(0))_{m\in\N}$ is uniformly bounded in $C^{2,\alpha}(\R^n)$. Further, noting that since $\int_Y v^{kl}_m = 0$ for all $m\in\N$, we have that $\lvert v_m^{kl}(0)\rvert = \left\lvert \int_Y (v_m^{kl}(0) - v_m^{kl})\right\rvert$ is uniformly bounded, and hence, $(v^{kl}_m)_{m\in\N}$ is uniformly bounded in $C^{2,\alpha}(\R^n)$. It is now quickly seen that $v^{kl}_m\rightarrow v^{kl}$ in $C^2(Y)$, where $v^{kl}$ is the solution to the (k,l)-th cell problem corresponding to $A$. We conclude that 
\begin{align*}
c_j^{kl}(A_m) = \int_Y r_m A_m e_j \cdot \nabla v^{kl}_m \longrightarrow \int_Y r A e_j \cdot \nabla v^{kl} =  c_j^{kl}(A)\quad\text{as}\quad m\rightarrow \infty,
\end{align*}
which is what we needed to show.
\end{proof}

\begin{proof}[Proof of denseness in Theorem \ref{Thm: density}]

Let $n\geq 2$ and $\alpha \in (0,1)$. We need to show that the set 
\begin{align*}
\{A\in C^{0,\alpha}(\T^n;\calS^n_+):A\text{ is type-$\eps$}\}
\end{align*}
is dense in $C^{0,\alpha}(\T^n;\calS^n_+)$. To this end, let $A^0\in C^{0,\alpha}(\T^n;\calS^n_+)$ be a type-$\eps^2$ map with invariant measure $r^0\in C^{0,\alpha}(\T^n;(0,\infty))$, and let $\delta > 0$. We need to show that there exists a type-$\eps$ map $A\in C^{0,\alpha}(\T^n;\calS^n_+)$ such that $\|A-A^0\|_{C^{0,\alpha}(\R^n)}\leq \delta$.

\medskip
\noindent \textbf{Step 1:} We suppose that 
\begin{align}\label{Pf of density: Ass Step1}
\int_Y r^0 A^0 e_1\cdot \nabla \fhi = 0\qquad \forall \fhi \in C^{\infty}(\T^n).
\end{align}
If this is not the case, skip Step 1 and set $A^1 := A^0$. Assuming \eqref{Pf of density: Ass Step1} holds, we are going to construct a map $A^1\in C^{0,\alpha}(\T^n;\calS^n_+)$ such that $\|A^1 - A^0\|_{C^{0,\alpha}(\R^n)}\leq \frac{\delta}{2}$ and such that there exists a function $q\in C^{\infty}(\T^n)$ with $\int_Y q = 0$ for which
\begin{align}\label{Pf of density: result Step1}
\int_Y r^1A^1 e_1\cdot \nabla q \neq 0,
\end{align}
where $r^1$ denotes the invariant measure of $A^1$. Let us introduce the map
\begin{align*}
P:= \frac{1}{r^0}Z\in C^{0,\alpha}(\T^n;\R^{n\times n}_{\mathrm{sym}}),
\end{align*}
where $Z\in C^{\infty}(\T^n;\R^{n\times n}_{\mathrm{sym}})$ is defined by
\begin{align*}
Z(y) = \zeta(y_1+y_2)\,\mathrm{diag}(1,-1,0,\dots,0)\quad\text{for}\quad y=(y_1,\dots,y_n)\in \R^n
\end{align*}
and $\zeta\in C^{\infty}(\T)$ is chosen such that $\zeta'\not\equiv 0$ and $\|P\|_{C^{0,\alpha}(\R^n)}\leq \frac{1}{2}$. We then define 
\begin{align*}
A^1:= A^0 + \delta P
\end{align*}
where we assume that $\delta$ is sufficiently small so that $A^1\in C^{0,\alpha}(\T^n;\calS^n_+)$. Note that $\|A^1 - A^0\|_{C^{0,\alpha}(\R^n)}\leq \frac{\delta}{2}$. We claim that the invariant measure $r^1\in C^{0,\alpha}(\T^n;(0,\infty))$ of the map $A^1$ is given by $r^1 = r^0$. Indeed, using that $r^0$ is the invariant measure of $A^0$, we have that
\begin{align*}
\int_Y r^0 (-A^1:D^2\fhi) = -\delta\int_Y r^0 P:D^2\fhi = -\delta\int_Y Z:D^2\fhi = -\delta\int_Y (D^2:Z)\fhi = 0
\end{align*}
for any $ \fhi\in C^{\infty}(\T^n)$. We now let $q\in C^{\infty}(\T^n)$ be such that $\int_Y q = 0$ and
\begin{align*}
\int_Y Z e_1\cdot \nabla q  \neq 0,
\end{align*}
which exists by the assumptions made on $\zeta$. Then, using $r^1=r^0$ and \eqref{Pf of density: Ass Step1}, we find that 
\begin{align*}
\int_Y r^1 A^1 e_1\cdot \nabla q = \delta\int_Y r^0  P e_1\cdot \nabla q =  \delta\int_Y Z e_1\cdot \nabla q \neq 0,
\end{align*}
i.e., \eqref{Pf of density: result Step1} holds.

\medskip
\noindent\textbf{Step 2:} We construct a type-$\eps$ map $A\in C^{0,\alpha}(\T^n;\calS^n_+)$ such that $\|A-A^1\|_{C^{0,\alpha}(\R^n)}\leq \frac{\delta}{2}$. Let us suppose that 
\begin{align}\label{Pf of density: ass Step2}
c_1^{11}(A^1) = \int_Y r^1 A^1 e_1\cdot \nabla v^{11}_{A^1} = 0,
\end{align}
as otherwise $A^1$ is type-$\eps$, in which case we set $A:=A^1$. Here, $v^{11}_{A^1}\in C^{2,\alpha}(\T^n)$ denotes the solution to the (1,1)-th cell problem corresponding to $A^1$. 

We set $\phi:= s q\in C^{\infty}(\T^n)$, where $q\in C^{\infty}(\T^n)$ is the function from Step 1 which satisfies $\int_Y q = 0$ and \eqref{Pf of density: result Step1}, and $s>0$ is chosen sufficiently small such that 
\begin{align*}
\gamma:=\frac{1}{1+A^1:D^2 \phi}\in C^{0,\alpha}(\T^n;(0,\infty))
\end{align*} 
and such that for
\begin{align*}
A:=\gamma A^1\in C^{0,\alpha}(\T^n;\calS^n_+)
\end{align*}
there holds $\|A-A^1\|_{C^{0,\alpha}(\R^n)}\leq \frac{\delta}{2}$. By Lemma \ref{Lmm: aB}, we have that
\begin{align}\label{Pf of density: cfromlmm}
c_1^{11}(A) = \bar{\gamma}\left( c_1^{11}(A^1) + \bar{a}_{11}^1 \int_Y r^1 A^1 e_1\cdot \nabla w\right),
\end{align}
where $w\in C^{2,\alpha}(\T^n)$ denotes the unique solution to
\begin{align}\label{Pf of density: wfromlmm}
-A:D^2 w = \gamma-\bar{\gamma}\quad\text{in }Y,\qquad w \text{ is }Y\text{-periodic},\qquad \int_Y w = 0,
\end{align}
and $\bar{a}_{11}^1,\bar{\gamma}>0$ denote the positive constants $\bar{a}_{11}^1:= \int_Y r^1 a^1_{11}$ and 
\begin{align}\label{Pf of density: gammabar=1}
\bar{\gamma}:= \left(\int_Y \frac{r^1}{\gamma}\right)^{-1} = \left(\int_Y (r^1+r^1 A^1:D^2\phi)\right)^{-1} = 1,
\end{align}
where we have used that $r^1$ is the invariant measure to $A^1$. We observe that the solution to \eqref{Pf of density: wfromlmm} is given by
\begin{align}\label{Pf of density: w=phi}
w = \phi = sq.
\end{align}
Indeed, $\phi\in C^{\infty}(\T^n)$, $\int_Y \phi = 0$, and we have 
\begin{align*}
-A:D^2 \phi = -\frac{A^1 :D^2 \phi}{1+A^1:D^2\phi} = \gamma - 1 = \gamma - \bar{\gamma}.
\end{align*}
We conclude from \eqref{Pf of density: cfromlmm} together with \eqref{Pf of density: ass Step2}, \eqref{Pf of density: gammabar=1} and \eqref{Pf of density: w=phi} that
\begin{align*}
c_1^{11}(A) = \bar{\gamma}\left( c_1^{11}(A^1) + \bar{a}_{11}^1 \int_Y r^1 A^1 e_1\cdot \nabla w\right) = s\,\bar{a}_{11}^1 \int_Y r^1 A^1 e_1\cdot \nabla q \neq 0,
\end{align*}
where we have used in the final step that \eqref{Pf of density: result Step1} holds. Therefore, $A$ is type-$\eps$.
\end{proof}

\section{Further analysis on type-$\eps^2$ and type-$\eps$ diffusion matrices} \label{sec:further}

In this section, we conduct a further systematic study on type-$\eps^2$ and type-$\eps$ diffusion matrices. 

\subsection{Perturbation arguments to create type-$\eps$ maps}\label{Subsec: Pertarg}
First, to demonstrate the perturbation argument from the proof of Theorem \ref{Thm: density} in a more explicit way, we consider a type-$\eps^2$ diagonal diffusion matrix with constant trace in dimension $n=2$ and prove the following result. 

\begin{lemma}[type-$\eps$ map via perturbation of type-$\eps^2$ map]\label{Thm: explicit perturbation}
Let $a\in C^{2,\alpha}(\T^2;(-1,1))$ for some $\alpha \in (0,1)$ and let $A\in C^{2,\alpha}(\T^2;\calS^2_+)$ be the type-$\eps^2$ map given by
\begin{align*}
A(y):=\mathrm{diag}(1+a(y),1-a(y))\quad\text{for}\quad y\in \R^2.
\end{align*}
Suppose that $\partial_{1} [r(1+a)] \not\equiv 0$, where $r\in C^{2,\alpha}(\T^2;(0,\infty))$ denotes the invariant measure to $A$. Let $w\in C^{4,\alpha}(\T^2)$ be the solution to
\begin{align}\label{wprob Lmm3.1}
-\Delta w = ra - \int_Y ra\quad\text{in }Y,\qquad w\text{ is $Y$-periodic},\qquad \int_{Y} w = 0,
\end{align}
and let $\phi:= s\, \partial_{1} w$ with $s>0$ chosen such that $\gamma:= \frac{1}{1+A:D^2 \phi} \in C^{0,\alpha}(\T^2;(0,\infty))$. Then, $\tilde{A}:=\gamma A$ is type-$\eps$.
\end{lemma}

\begin{proof}
First, we note that $A$ is indeed type-$\eps^2$ by Theorem \ref{Thm: Main}. 
Let us write
\begin{align}\label{Pf of Thm 1.4: smallabar}
\bar{a}:=\int_Y ra,
\end{align}
let $w\in C^{4,\alpha}(\T^2)$ be the unique solution to \eqref{wprob Lmm3.1}, and let $\phi:= s\, \partial_{1} w$ with $s>0$ chosen such that 
\begin{align*}
\gamma:= \frac{1}{1+A:D^2 \phi} \in C^{0,\alpha}(\T^2;(0,\infty)).
\end{align*}
We need to show that 
\begin{align*}
\tilde{A}:=\gamma A\in C^{0,\alpha}(\T^2;\calS^2_+)
\end{align*}
is type-$\eps$.

\medskip
\noindent \textbf{Step 1:} We claim that the invariant measure $r$ of $A$ and the solution $v^{kl}\in C^{4,\alpha}(\T^2)$ to the (k,l)-th cell problem \eqref{vij problem} for $k,l\in\{1,2\}$ are given by
\begin{align}\label{Pf of Thm 1.4: Step 1 goal}
r = 1 + \partial_{11}^2 w - \partial_{22}^2 w,\qquad v^{11} = - v^{22} =  w,\quad v^{12}=v^{21}\equiv 0.
\end{align}
To this end, let us start by noting that $A$ is of the form $A(y) = C + a(y) M$ for $y\in \R^2$ with $C:=I_2\in \calS^2_+$ and $M:=\mathrm{diag}(1,-1)\in \R^{2\times 2}_{\mathrm{sym}}$. We then know from the proof of Theorem \ref{Thm: Main} (see \eqref{Pf Thm 1.1: r=xi} and \eqref{Pf of Thm 1.1: vkl=etakl}) that
\begin{align}\label{Pf of Thm 1.4: r,vkl in terms of what}
r = 1+M:D^2 \hat{w},\qquad v^{kl} = m_{kl} \hat{w}
\end{align}
for $k,l\in\{1,2\}$, where $\hat{w}\in C^{4,\alpha}(\T^2)$ is the unique solution to (compare with \eqref{Pf Thm 1.1: w})
\begin{align*}
-A:D^2 \hat{w} = a-\bar{a}  \quad\text{in }Y,\qquad \hat{w} \text{ is }Y\text{-periodic},\qquad \int_Y \hat{w} = 0.
\end{align*} 
Note that, by \eqref{Pf of Thm 1.1: -C:D^2w} with $C=I_2$, we have that $\hat{w}$ satisfies
\begin{align*}
-\Delta \hat{w} = ra-\bar{a}\quad\text{in }Y,\qquad \hat{w} \text{ is }Y\text{-periodic},\qquad \int_Y \hat{w} = 0,
\end{align*} 
and thus, in view of \eqref{wprob Lmm3.1}, \eqref{Pf of Thm 1.4: smallabar} and using uniqueness of solutions, we have $\hat{w}=w$. Using this, we see that \eqref{Pf of Thm 1.4: Step 1 goal} follows from \eqref{Pf of Thm 1.4: r,vkl in terms of what}.

\medskip
\noindent \textbf{Step 2:} We apply Lemma \ref{Lmm: aB} to find that
\begin{align*}
c_1^{11}(\tilde{A}) = \bar{\gamma}\left( c_1^{11}(A) + \bar{a}_{11} \int_Y r A e_1\cdot \nabla w_{\gamma}\right),
\end{align*}
where $w_{\gamma}\in C^{2,\alpha}(\T^2)$ denotes the unique solution to
\begin{align}\label{Pf of Thm 1.4: wfromlmm}
-\tilde{A}:D^2 w_{\gamma} = \gamma-\bar{\gamma}\quad\text{in }Y,\qquad w_{\gamma} \text{ is }Y\text{-periodic},\qquad \int_Y w_{\gamma} = 0,
\end{align}
and $\bar{\gamma}>0$ is given by 
\begin{align}\label{Pf of Thm 1.4: gammabar=1}
\bar{\gamma}:= \left(\int_Y \frac{r}{\gamma}\right)^{-1} = \left(\int_Y (r+r A:D^2\phi)\right)^{-1} = 1,
\end{align}
where we have used that $r$ is the invariant measure to $A$. Using that $c_1^{11}(A) = 0$ since $A$ is type-$\eps^2$, \eqref{Pf of Thm 1.4: gammabar=1} and $\bar{a}_{11} = 1+\bar{a}$, we find that
\begin{align}\label{Pf of Thm 1.4: st2res}
c_1^{11}(\tilde{A}) =(1+\bar{a}) \int_Y r A e_1\cdot \nabla w_{\gamma} = (1+\bar{a}) \int_Y r (1+a) \partial_1  w_{\gamma}.
\end{align}
\textbf{Step 3:} We claim that the solution $w_{\gamma}$ of \eqref{Pf of Thm 1.4: wfromlmm} is given by
\begin{align}\label{Pf of Thm 1.4: st3res}
w_{\gamma} = \phi = s\,\partial_1 w.
\end{align}
Indeed, $\phi$ satisfies $\phi \in C^{2,\alpha}(\T^2)$, $\int_Y \phi = 0$, and we have that
\begin{align*}
-\tilde{A}:D^2 \phi = -\frac{A :D^2 \phi}{1+A:D^2\phi} = \gamma - 1 = \gamma - \bar{\gamma}.
\end{align*}
It follows that $w_{\gamma} = \phi$.

\medskip
\noindent \textbf{Step 4:} From the results \eqref{Pf of Thm 1.4: st2res} and \eqref{Pf of Thm 1.4: st3res} from Steps 2 and 3, we deduce that
\begin{align*}
c_1^{11}(\tilde{A}) = s(1+\bar{a}) \int_Y r (1+a) \partial_{11}^2  w = -s(1+\bar{a}) \int_Y (\partial_1[r (1+a)])(\partial_1 w).
\end{align*}
Noting that by \eqref{wprob Lmm3.1} and \eqref{Pf of Thm 1.4: Step 1 goal} we have
\begin{align}\label{Pf of Thm 1.4 d1(r(1+a))}
\partial_1[r(1+a)] = \partial_1[ 1 + \partial_{11}^2 w -   \partial_{22}^2 w - \Delta w + \bar{a} ] = -2\, \partial_{122}^3 w,
\end{align}
we deduce that there holds
\begin{align}\label{Pf of Thm 1.4: c111Ati 1}
c_1^{11}(\tilde{A}) = 2s(1+\bar{a})\int_Y (\partial_{122}^3 w)(\partial_1 w).
\end{align} 
Finally, we observe that 
\begin{align}\label{Pf of Thm 1.4: c111Ati 2}
\int_Y (\partial_{122}^3 w)(\partial_1 w) = -\int_Y (\partial_{12}^2 w)^2 \neq 0
\end{align}
which holds as $\partial_{12}^2 w\not\equiv 0$ since $\frac{\mathrm{d}}{\mathrm{d}y_2}[ \partial_{12}^2 w ] = -\frac{1}{2}\partial_1[r(1+a)] \not\equiv 0$ by \eqref{Pf of Thm 1.4 d1(r(1+a))}. Combining \eqref{Pf of Thm 1.4: c111Ati 1} and 
\eqref{Pf of Thm 1.4: c111Ati 2}, and using that $s>0$ and $1+\bar{a}>0$, we find that  
\begin{align*}
c_1^{11}(\tilde{A}) \neq 0,
\end{align*}
and hence, $\tilde{A}$ is type-$\eps$.
\end{proof}

Let us point out that Lemma \ref{Thm: explicit perturbation} can also be proved quickly using Theorem \ref{Thm: Complete char}.

\subsection{Classification of $A$ for which $r$ and $rA$ are of special structure}\label{Subsec: 3.2}

The following result enables us to classify a diagonal map $A\in C^{0,\alpha}(\T^2;\calS^2_+)$ as type-$\eps^2$ or type-$\eps$ when the map $rA$ has constant trace and the invariant measure $r$ has a special structure.  
\begin{lemma}[Classification when $r,rA$ are of special structure in $n=2$]\label{Thm: r,a special}
Let $A\in C^{0,\alpha}(\T^2;\calS^2_+)$ for some $\alpha \in (0,1)$ with invariant measure $r\in C^{0,\alpha}(\T^2;(0,\infty))$. Suppose that for some $r_1,r_2\in C^{0,\alpha}(\T)$ we have 
\begin{align*}
r(y_1,y_2) = r_1(y_1+y_2) + r_2(y_1 - y_2)
\end{align*}
for any $(y_1,y_2)\in \R^2$, and that for some $c>0$ and $a\in C^{0,\alpha}(\T^2;(0,c))$  we have 
\begin{align*}
A(y) = \frac{1}{r(y)}\,\mathrm{diag}(a(y),c-a(y))
\end{align*}
for any $y\in \R^2$. Then, $A$ is type-$\eps^2$ if and only if
\begin{align*}
\int_0^1 \int_0^1 a(y_1,y_2) R_1'(y_1+y_2)\,\mathrm{d}y_1\,\mathrm{d}y_2 = \int_0^1 \int_0^1 a(y_1,y_2) R_2'(y_1-y_2)\,\mathrm{d}y_1\,\mathrm{d}y_2 = 0,
\end{align*}
where $R_i\in C^{2,\alpha}(\T)$ is such that $R_i'' = r_i - \int_{(0,1)} r_i$ and $\int_{(0,1)} R_i = 0$ for $i\in \{1,2\}$.
\end{lemma}

\begin{proof}
First, let us note that with $b:=\frac{2}{c}a-1\in C^{0,\alpha}(\T^2;(-1,1))$ we have that
\begin{align*}
A(y) = \frac{c}{2}\frac{1}{r(y)}B(y),\qquad B(y):= \mathrm{diag}(1+b(y),1-b(y))
\end{align*}
for any $y\in \R^2$. We observe that the invariant measure $r_B$ of $B=\frac{2}{c}(rA)$ is given by $r_B\equiv 1$ since $r$ is the invariant measure of $A$.

\medskip
\noindent \textbf{Step 1:} By Theorem \ref{Thm: Complete char}, the map $A$ is type-$\eps^2$ if and only if
\begin{align}\label{Pf Lemma3.2: to show}
\int_Y (\partial_1 w_A)(\partial_{22}^2 w_B) = \int_Y (\partial_2 w_A)(\partial_{11}^2 w_B) = 0,
\end{align}
where $w_A\in C^{2,\alpha}(\T^2)$ denotes the unique solution to
\begin{align}\label{Pf Lemma3.2: wA}
-A:D^2 w_A = \frac{c}{2}\left(\frac{1}{r}-1\right) \quad\text{in }Y,\qquad w_A \text{ is }Y\text{-periodic},\qquad \int_Y w_A = 0,
\end{align}
and $w_B\in C^{2,\alpha}(\T^2)$ denotes the unique solution to
\begin{align*}
-\Delta w_B = b - \int_Y b = \frac{2}{c}\left(a-\int_Y a\right)\quad\text{in }Y,\qquad w_B \text{ is }Y\text{-periodic},\qquad \int_Y w_B = 0.
\end{align*}
Further, in view of the proof of Theorem \ref{Thm: Complete char} (see \eqref{Pf of Thm compl char: rB}), the function $w_B$ satisfies $\partial_{11}^2 w_B - \partial_{22}^2 w_B = r_B - 1 \equiv 0$ and hence, we have that
\begin{align}\label{Pf Lemma 3.2: result St1}
\partial_{11}^2 w_B = \partial_{22}^2 w_B = \frac{1}{2}\Delta w_B = -\frac{1}{c}\left(a-\int_Y a\right).
\end{align}

\medskip
\noindent \textbf{Step 2:} We claim that the solution $w_A$ to \eqref{Pf Lemma3.2: wA} is given by
\begin{align}\label{Pf Lemma3.2: result St2}
w_A(y_1,y_2) = \frac{1}{2}\left(R_1(y_1+y_2) + R_2(y_1-y_2)\right)
\end{align}
for any $(y_1,y_2)\in \R^2$. To this end, note that $w_A\in C^{2,\alpha}(\T^2)$ and $\int_Y w_A = 0$. Further, noting that $\partial_{11}^2 w_A = \partial_{22}^2 w_A = \frac{1}{2}(r-1)$, we have 
\begin{align*}
-A:D^2 w_A = -\frac{c}{2r}B:D^2 w_A = -\frac{c}{r}\partial_{11}^2 w_A = -\frac{c}{2r}(r - 1) = \frac{c}{2}\left(\frac{1}{r}-1\right).
\end{align*}

\medskip
\noindent \textbf{Step 3:} Let us define the values $Q_1,Q_2\in \R$ by
\begin{align}\label{Q1,Q2}
\begin{split}
Q_1&:=\int_0^1 \int_0^1 a(y_1,y_2) R_1'(y_1+y_2)\mathrm{d}y_1\mathrm{d}y_2,\\ Q_2&:=\int_0^1 \int_0^1 a(y_1,y_2) R_2'(y_1-y_2)\mathrm{d}y_1\mathrm{d}y_2.
\end{split}
\end{align}
In view of \eqref{Pf Lemma3.2: to show}, and using the results \eqref{Pf Lemma 3.2: result St1} and \eqref{Pf Lemma3.2: result St2} from the previous steps, we have that $A$ is type-$\eps^2$ if and only if
\begin{align*}
\int_Y (\partial_1 w_A)(\partial_{22}^2 w_B) = -\frac{1}{2c}(Q_1+Q_2) = 0,\quad
 \int_Y (\partial_2 w_A)(\partial_{11}^2 w_B)&= \frac{1}{2c}(Q_2-Q_1) = 0. 
\end{align*}
We conclude that $A$ is type-$\eps^2$ if and only if $Q_1=Q_2=0$, which is what we needed to show. Finally, let us note that in view of Step 2 of the proof of Theorem \ref{Thm: Complete char}, we can explicitly compute that
\begin{align}\label{Pf Lemma3.2:cjkl}
\begin{split}
c_1^{11}(A) &= \frac{\int_Y a}{c-\int_Y a}c_1^{22}(A) =  \frac{\int_Y a}{c}(Q_1+Q_2),\\ c_2^{11}(A) &= \frac{\int_Y a}{c-\int_Y a}c_2^{22}(A) =  \frac{\int_Y a}{c}(Q_2-Q_1),
\end{split}
\end{align}
and that $c_j^{kl}(A) = 0$ when $k\neq l$.
\end{proof} 

\begin{remark}[Application of Lemma \ref{Thm: r,a special}]\label{Rk: Example recover ST}
Let us demonstrate an application of Lemma \ref{Thm: r,a special} by considering the map $A\in C^{\infty}(\mathbb{T}^2;\calS^2_+)$ defined in \eqref{cbad from ST}, which we can write as
\begin{align*}
A(y)=\frac{1}{r(y)}\, \mathrm{diag}\left(a(y),2-a(y)\right)\quad\text{for}\quad y\in \R^2,
\end{align*}
where $r,a:\R^2\rightarrow \R$ are defined as
\begin{align*}
r(y_1,y_2)&:= 1+\frac{1}{4}(\cos(2\pi y_1)-2\sin(2\pi y_1))\sin(2\pi y_2) = r_1(y_1+y_2) + r_2(y_1 - y_2),\\
a(y_1,y_2)&:= 1-\frac{1}{2}\sin(2\pi y_1)\sin(2\pi y_2) 
\end{align*}
for $(y_1,y_2)\in \R^2$, where $r_1,r_2\in C^{\infty}(\T)$ denote the functions given by
\begin{align*}
r_1(t):=1-r_2(t):= \frac{1}{8}\left(\sin(2\pi t)+2\cos(2\pi t)\right)\quad\text{for}\quad t\in \R.
\end{align*}
Note that $r$ is the invariant measure of $A$. The functions $R_i\in C^{\infty}(\T)$ satisfying $R_i'' = r_i - \int_{(0,1)} r_i$ and $\int_{(0,1)} R_i = 0$ for $i\in \{1,2\}$ are given by
\begin{align*}
R_1(t) := -R_2(t):=-\frac{1}{32\pi^2}\left(\sin(2\pi t)+2\cos(2\pi t) \right)\quad\text{for}\quad t\in \R.
\end{align*}
Computing the values $Q_1,Q_2$ defined in \eqref{Q1,Q2} yields $Q_1 = Q_2 = -\frac{1}{128\pi}$ and we deduce from Lemma \ref{Thm: r,a special} that $A$ is type-$\eps$. Moreover, in view of \eqref{Pf Lemma3.2:cjkl} and noting that in this situation we have $\int_Y a = 1$ and $c:=2$, we compute
\begin{align*}
c_1^{11}(A) = c_1^{22}(A) = -\frac{1}{128\pi},\qquad c_2^{11}(A) = c_2^{22}(A) = 0,
\end{align*}
and $c_j^{kl}(A) = 0$ when $k\neq l$. We have recovered the result from \cite[Theorem 1.4]{ST21}.
\end{remark}

\subsection{Diffusion matrices with constant trace in dimension $n=2$}\label{Subsec: 3.new}

In this subsection, we consider maps $A\in C^{0,\alpha}(\T^2;\calS_+^2)$ with constant trace. First, we observe that such a map $A$ is a type-$\eps^2$ diffusion matrix if and only if the third-order homogenized tensor $(c_j^{kl}(A))_{1\leq j,k,l\leq 2}$ vanishes.

\begin{lemma}\label{lem: trace 1 in 2D}
Let $A\in C^{0,\alpha}(\T^2;\calS_+^2)$ for some $\alpha \in (0,1)$, and suppose that $\tr(A) \equiv c$ for some constant $c>0$. Then, $A$ is type-$\eps^2$ if and only if $c_j^{kl}(A)=0$ for all $j,k,l\in\{1,2\}$. 
\end{lemma}

\begin{proof}
Let $A\in C^{0,\alpha}(\T^2;\calS_+^2)$ with $\alpha \in (0,1)$ and suppose that $\tr(A) \equiv c$ for some constant $c>0$. We denote the invariant measure of $A$ by $r\in C^{0,\alpha}(\T^2;(0,\infty))$, the effective coefficient corresponding to $A$ by $\bar{A}:=\int_Y rA\in \calS^2_+$, and the solution to the (k,l)-th cell problem \eqref{vij problem} corresponding to $A$ by $v^{kl}\in C^{2,\alpha}(\T^2)$ for $k,l\in\{1,2\}$. Let us observe that 
\begin{align}\label{v11+v22 zero}
v^{11} + v^{22}\equiv 0.
\end{align}
Indeed, $v^{11} + v^{22}\in C^{2,\alpha}(\T^2)$, we have $\int_Y (v^{11}+v^{22}) = 0$, and there holds
\begin{align*}
-A:D^2 (v^{11}+v^{22}) = \left(a_{11} - \bar{a}_{11}\right)+\left(a_{22} - \bar{a}_{22}\right) = \mathrm{tr}(A) - \mathrm{tr}(\bar{A})\equiv 0,
\end{align*}
where we have used that $\mathrm{tr}(A)\equiv c$ and $\mathrm{tr}(\bar{A}) = \int_Y \mathrm{tr}(rA) = c$. It follows that \eqref{v11+v22 zero} holds, and we deduce from the definition of the values $c_j^{kl}(A)$ that
\begin{align}\label{eq:tr1-2}
c^{11}_j(A)+c^{22}_j(A) =   \int_Y rAe_j\cdot \nabla (v^{11}+v^{22}) = 0\qquad \forall j\in \{1,2\}.
\end{align}
Since by Definition \ref{Def: cg cb}, $A$ is type-$\eps^2$ if and only if
\begin{align*}
c_1^{11}(A)=c_2^{22}(A)=c_2^{11}(A)+2\,c_1^{12}(A)=c_1^{22}(A)+2\,c_2^{12}(A)=0,
\end{align*} 
we deduce from \eqref{eq:tr1-2} that $A$ is a type-$\eps^2$ if and only if $c_j^{kl}(A) = 0$ for all $j,k,l\in\{1,2\}$, which is what we needed to show.
\end{proof}

We are now in a position to show that a diagonal trace-one map $A\in C^{0,\alpha}(\T^2;\calS_+^2)$ (which, by Corollary \ref{Cor: Cor1}, is a type-$\eps^2$ diffusion matrix) is the limit of a sequence of trace-one type-$\eps$ diffusion matrices in $C^{0,\alpha}(\T^2;\calS_+^2)$.

\begin{lemma}\label{lem: c-bad trace 1 in 2D}
Let $A\in C^{0,\alpha}(\T^2;\calS_+^2)$ with $\alpha \in (0,1)$ be a diagonal map which satisfies $\tr(A) \equiv 1$.
Then, there exists a sequence of type-$\eps$ diffusion matrices $(A_m)_{m\in\N} \subset C^{0,\alpha}(\T^2;\calS_+^2)$ such that $\tr(A_m) \equiv 1$ and $\|A_m-A\|_{C^{0,\alpha}(\R^2)}\rightarrow 0$ as $m\rightarrow \infty$. 
\end{lemma}

\begin{proof}
Let $A\in C^{0,\alpha}(\T^2;\calS_+^2)$ for some $\alpha \in (0,1)$, and suppose that $A$ is diagonal and satisfies $\tr(A) \equiv 1$. For each $m\in \N$, we will show that there exists a type-$\eps$ diffusion matrix $A_m\in C^{0,\alpha}(\T^2;\calS_+^2)$ such that $\tr(A_m) \equiv 1$ and $\|A_m-A\|_{C^{0,\alpha}(\R^2)} \leq \frac{1}{m}$.

First, we observe that by the hypothesis and Corollary \ref{Cor: Cor1}, $A$ is a type-$\eps^2$ diffusion matrix.
We deduce from the proof of Theorem \ref{Thm: density} that there exists a diagonal type-$\eps$ diffusion matrix $B = \mathrm{diag}(b_1,b_2) \in C^{0,\alpha}(\T^2;\calS_+^2)$ such that $c_1^{11}(B)\neq 0$ and 
\begin{align*}
\|B-A\|_{C^{0,\alpha}(\R^2)} \leq \frac{1}{4m}.
\end{align*}
As the map $c_1^{11}:C^{0,\alpha}(\T^2;\calS_+^2)\rightarrow \R$ is continuous, there exists a constant $\delta \in (0,\frac{1}{8m})$ sufficiently small such that for $\tilde{B}:\R^2\rightarrow \R^{2\times 2}_{\mathrm{sym}}$ given by
\begin{align*}
\tilde{B}(y) := \begin{pmatrix}
b_{1}(y) & \delta\\
\delta & b_{2}(y)
\end{pmatrix}\quad\text{for}\quad y\in \R^2,
\end{align*}
we have that $\tilde{B}\in C^{0,\alpha}(\T^2;\calS_+^2)$ and $c_1^{11}(\tilde{B}) \neq 0$. We define the map
\begin{align*}
A_m:=\frac{1}{\tr(\tilde B)} \tilde B \in C^{0,\alpha}(\T^2;\calS_+^2)
\end{align*}
and observe that $\tr(A_m) \equiv 1$ and $\|A_m-A\|_{C^{0,\alpha}(\R^2)} \leq \frac{1}{m}$.
We claim that $A_m$ is type-$\eps$.
Suppose that this were not true, i.e., $A_m$ is type-$\eps^2$. Then, by Lemma \ref{lem: trace 1 in 2D}, we have that $c_j^{kl}(A_m)=0$ for all $j,k,l\in \{1,2\}$. Noting that we can write 
\begin{align*}
\tilde B =(b_1+b_2)A_m =  \frac{1}{C:A_m} A_m,\quad\text{where}\quad C:=\begin{pmatrix}
0 & \frac{1}{2\delta}\\
\frac{1}{2\delta} & 0
\end{pmatrix},
\end{align*}
we deduce from Theorem \ref{Thm: 1/C:A A} that $\tilde{B}$ is type-$\eps^2$; a contradiction to $c_1^{11}(\tilde{B})\neq 0$. The proof is complete.
\end{proof}

Let us provide an explicit example of a type-$\eps$ diffusion matrix with constant trace in dimension $n=2$:

\begin{remark}[type-$\eps$ constant-trace map in $n=2$]\label{Rk: type-eps const tr in n=2}
The map $A\in C^{\infty}(\T^2;\calS^2_+)$ defined by
\begin{align*}
A(y_1,y_2):=\begin{pmatrix}
5+ \sin(2\pi y_1) & 1+\cos(2\pi y_1)\\
1+\cos(2\pi y_1) & 5-\sin(2\pi y_1)
\end{pmatrix}\quad\text{for}\quad (y_1,y_2)\in \R^2
\end{align*}
is a type-$\eps$ diffusion matrix. Indeed, one can show that (we omit the details)
\begin{align*}
c_2^{12}(A) = \int_0^1 \left(\frac{1}{2\pi}-\frac{\sqrt{6}}{\pi}\frac{1+ \cos(2\pi t)}{5+\sin(2\pi t)}\right)\ln(5+\sin(2\pi t))\,\mathrm{d}t = 0.003\dots\neq 0,
\end{align*}
and hence, by Lemma \ref{lem: trace 1 in 2D}, we have that $A$ is type-$\eps$. 
\end{remark}

\subsection{Scalar multiples and sums of type-$\eps^2$ maps}\label{Subsec: 3.3}

It is natural to ask whether scalar multiplication and addition are type-$\eps^2$-preserving operations. Regarding scalar multiplication, we have a positive answer.

\begin{lemma}[Scalar multiples of a type-$\eps^2$ map are type-$\eps^2$]\label{Thm: scalar multiples}
Let $A\in C^{0,\alpha}(\T^n;\calS_+^n)$ for some $\alpha \in (0,1)$, and let $\lambda >0$. Then, we have that
\begin{align}\label{cjkl(lA)=lcjkl(A)}
c_j^{kl}(\lambda A) = \lambda\, c_j^{kl}(A)\qquad \forall j,k,l\in \{1,\dots,n\}.
\end{align}
In particular, $\lambda A$ is type-$\eps^2$ if and only if $A$ is type-$\eps^2$. 
\end{lemma}

\begin{proof}
Let $A\in C^{0,\alpha}(\T^n;\calS_+^n)$ for some $\alpha \in (0,1)$, and let $\lambda >0$. We denote the invariant measure of $A$ by $r\in C^{0,\alpha}(\T^n;(0,\infty))$ and the effective coefficient to $A$ by $\bar{A}:=\int_Y rA \in \calS^n_+$. By Lemma \ref{Lmm: aB}, and noting that $\bar{\lambda}:=\int_Y \lambda r = \lambda$, we have that
\begin{align*}
c_j^{kl}(\lambda A) = \lambda\, c_j^{kl}(A)+ \lambda\,\bar{a}_{kl} \int_Y r A e_j\cdot \nabla w\qquad \forall j,k,l\in \{1,\dots,n\},
\end{align*}
where $w$ is the unique function in $C^{2,\alpha}(\T^n)$ satisfying $-\lambda A:D^2 w \equiv \lambda-\bar{\lambda} = 0$ and $\int_Y w = 0$. Noting that $w\equiv 0$, we find that \eqref{cjkl(lA)=lcjkl(A)} holds. 
\end{proof}

Regarding addition, we have a negative answer. Even a simple addition with the identity matrix can turn a type-$\eps^2$ map into a type-$\eps$ one.

\begin{remark}[Existence of type-$\eps^2$ map $A$ for which $A+I_n$ is type-$\eps$]\label{Rk: A+I}
There exists $A\in C^{\infty}(\T^2;\calS^2_+)$ such that $A$ is type-$\eps^2$ and $\tilde{A}\in C^{\infty}(\T^2;\calS^2_+)$ given by $\tilde{A}(y) = A(y) +I_2$ for $y\in \R^2$ is type-$\eps$; see Section \ref{Subsec: A+I}. 
\end{remark}

\subsection{The case of constant invariant measure}\label{Subsec: 3.4}

In this subsection, we consider diffusion matrices $A\in C^{0,\alpha}(\T^n;\calS^n_+)$ with constant invariant measure, i.e., $r\equiv 1$. The case $r\equiv 1$ is of particular interest as, in this case, the nondivergence-form operator $L^{\eps}:= -A(\frac{\cdot}{\eps}):D^2$ can be rewritten as a divergence-form operator $L^{\eps} = -\nabla\cdot B(\frac{\cdot}{\eps})\nabla $ for some uniformly elliptic $B\in C^{0,\alpha}(\T^n;\R^{n\times n})$ (which, in general, is not symmetric); see \cite{AL89}. 

As it turns out, $r\equiv 1$ is not a sufficient condition to guarantee that $A$ is type-$\eps^2$, not even in dimension $n=2$ when additionally assuming that $A$ is diagonal. 

\begin{remark}[type-$\eps$ diagonal map with constant invariant measure]\label{Rk: r=1 c-bad}
There exists a map $A\in C^{\infty}(\T^2;\calS^2_+)$ of the form $A = \mathrm{diag}(a_1,a_2)$ for some $a_1,a_2\in C^{\infty}(\T^2;(0,\infty))$ such that the invariant measure of $A$ is $r\equiv 1$ and $A$ is type-$\eps$; see Section \ref{Subsec: c-bad diagonal r=1 n=2}.
\end{remark}

Note that, given a map $A\in C^{0,\alpha}(\T^n;\calS^n_+)$ with invariant measure $r$, we have that the map $rA$ has constant invariant measure. One might expect that the product of a type-$\eps^2$ map and its invariant measure is again a type-$\eps^2$ map. However, this is not the case, not even in dimension $n=2$.

\begin{remark}[Existence of a type-$\eps^2$ map $A$ for which $rA$ is type-$\eps$]\label{Rk: Mult by r}
There exists a type-$\eps^2$ map $A\in C^{\infty}(\T^2;\calS^2_+)$ for which $rA\in C^{\infty}(\T^2;\calS^2_+)$ is type-$\eps$, where $r$ denotes the invariant measure of $A$; see Section \ref{Subsec: c-bad diagonal r=1 n=2}.
\end{remark} 

\subsection{Some sufficient conditions for $A$ to be type-$\eps^2$}\label{Subsec: 3.5}

Some structural assumptions on $A\in C^{0,\alpha}(\mathbb{T}^n;\calS^n_+)$ are quickly seen to guarantee that $A$ is type-$\eps^2$. 
\begin{remark}[Sufficient conditions for $A$ to be type-$\eps^2$]\label{Rk: Suff cond for c-good}
Let $A\in C^{0,\alpha}(\mathbb{T}^n;\calS^n_+)$ for some $\alpha \in (0,1)$. Then, the following assertions hold.
\begin{itemize}
\item[(i)] If $A$ is of the form
\begin{align*}
A(y) = \mathrm{diag}(a_1(y_1),a_2(y_2),\dots,a_n(y_n))\quad\text{for}\quad y=(y_1,\dots,y_n)\in \R^n
\end{align*}
for some $a_1,\dots,a_n\in C^{0,\alpha}(\T;(0,\infty))$, then $A$ is type-$\eps^2$. This follows from the fact that $\mathrm{div}(rA) = 0$ (weakly) as the invariant measure $r$ is given by
\begin{align*}
r(y) = \prod_{i=1}^n \left(\int_0^1 \frac{\mathrm{d}t}{a_i(t)}\right)^{-1} \frac{1}{a_i(y_i)}\quad\text{for}\quad y=(y_1,\dots,y_n)\in \R^n.
\end{align*}
\item[(ii)] If $A$ is of the form
\begin{align*}
A(y) = \mathrm{diag}(a_1(y),\dots,a_n(y))\quad\text{for}\quad y\in \R^n
\end{align*}
for some $a_1,\dots,a_n\in C^{0,\alpha}(\mathbb{T}^n;(0,\infty))$ such that, for any $i\in \{1,\dots,n\}$, the function $a_i$ is independent of $y_i$, then $A$ is type-$\eps^2$. This follows from the fact that $\mathrm{div}(rA) = 0$ (weakly) as the invariant measure $r$ is given by $r\equiv 1$.
\item[(iii)] If $A$ is of the form
\begin{align*}
A(y) = \tilde{A}(y_1)\quad\text{for}\quad y=(y_1,\dots,y_n)\in \R^n
\end{align*}
for some $\tilde{A}\in C^{0,\alpha}(\mathbb{T};\calS_+^n)$ with $\tilde{a}_{1j} \equiv 0$ for all $2\leq j\leq n$, then $A$ is type-$\eps^2$. This follows from the fact that $\mathrm{div}(rA) = 0$ (weakly) as the invariant measure $r$ is given by
\begin{align*}
r(y) = \left(\int_0^1 \frac{\mathrm{d}t}{a_{11}(t)}\right)^{-1}\frac{1}{\tilde{a}_{11}(y_1)}\quad\text{for}\quad y=(y_1,\dots,y_n)\in \R^n.
\end{align*}
\item[(iv)] If $A$ is a shifted even function, i.e., if for some $x\in \R^n$ there holds $A(x-y) = A(x+y)$ for all $y\in \R^n$, then $A$ is type-$\eps^2$.
\end{itemize}
\end{remark}
The results presented in Remark \ref{Rk: Suff cond for c-good} are generalizations of \cite[Corollary 1.6]{GTY20} and \cite[Theorem 1.5(c)]{GTY20}.
\begin{remark}
In \cite[Corollary 1.6]{GTY20}, it is claimed that any $A\in C^{2}(\mathbb{T}^n;\calS^n_+)$ with $A(y) = \tilde{A}(y_1)$ for some $\tilde{A}\in C^{2}(\T;\calS_+^n)$ is type-$\eps^2$. This is not true in general: the diffusion matrix from Remark \ref{Rk: type-eps const tr in n=2} provides a counterexample. The condition $\tilde{a}_{1j} \equiv 0$ for all $2\leq j\leq n$ needs to be added.
\end{remark}

\section{Collection of Counterexamples}\label{Sec: 4}

In this section, we collect various counterexamples to demonstrate the results of Remark \ref{Rk: Cor of Thm 1.1 not true in n>2} (see Sections \ref{Subsec: c-bad diagonal map with constant trace in $n=3$}--\ref{Subsec: Example with optimal rate Oeps}) and Remarks \ref{Rk: A+I}--\ref{Rk: Mult by r} (see Sections \ref{Subsec: A+I}--\ref{Subsec: c-bad diagonal r=1 n=2}).

\subsection{type-$\eps$ diagonal map with constant trace in $n=3$}\label{Subsec: c-bad diagonal map with constant trace in $n=3$}

In this section, we show that there exists a type-$\eps$ map $A\in C^{\infty}(\T^3;\calS^3_+)$ of the form 
\begin{align*}
A(y) = \mathrm{diag}(a_1(y),a_2(y),a_3(y))\quad\text{for}\quad y\in \R^3
\end{align*}
for some functions $a_1,a_2,a_3\in C^{\infty}(\T^3;(0,\infty))$ satisfying that $a_1+a_2+a_3$ is constant.

\begin{theorem}[Construction of type-$\eps$ maps with constant trace in $n = 3$]\label{Thm: construct c-bad trace}
Let $b_1,b_2\in C^{0,\alpha}(\T^2;(0,\infty))$ with $\alpha \in (0,1)$ be chosen such that the map
\begin{align*}
B:=\mathrm{diag}(b_1,b_2)\in C^{0,\alpha}(\T^2;\calS^2_+)
\end{align*}
satisfies $c_1^{11}(B) \neq 0$. Further, let $c > \|b_1+b_2\|_{\infty}$ be a chosen constant and set $b_3 := c - b_1 - b_2\in C^{0,\alpha}(\T^2;(0,\infty))$. Then, the map 
\begin{align*}
A := \mathrm{diag}(a_1,a_2,a_3)\in C^{0,\alpha}(\T^3;\calS^3_+)
\end{align*}
with $a_1,a_2,a_3\in C^{0,\alpha}(\T^3;(0,\infty))$ defined by
\begin{align*}
a_i(y):=b_i(y_1,y_2)\quad\text{for}\quad y=(y_1,y_2,y_3)\in \R^3
\end{align*}
for $i\in \{1,2,3\}$ is type-$\eps$.
\end{theorem}

\begin{proof}
Let $r_B\in C^{0,\alpha}(\T^2;(0,\infty))$ denote the invariant measure of $B$ and let $v^{11}_B\in C^{2,\alpha}(\T^2)$ denote the solution to the (1,1)-th cell problem corresponding to $B$. First, we claim that the function $r_A\in C^{0,\alpha}(\T^3;(0,\infty))$ defined by
\begin{align*}
r_A(y) := r_B(y_1,y_2)\quad\text{for}\quad y=(y_1,y_2,y_3)\in \R^3
\end{align*}
is the invariant measure of $A$. Indeed, note that $\int_{[0,1]^3} r_A = \int_{[0,1]^2} r_B = 1$, and for any $\fhi\in C^{\infty}(\T^3)$ there holds 
\begin{align*}
\int_{[0,1]^3} r_A (-A:D^2 \fhi) = \int_{[0,1]^2} r_B (-B:D^2 \phi) = 0,
\end{align*}
where $\phi\in C^{\infty}(\T^2)$ denotes the function defined as $\phi(y_1,y_2):= \int_0^1 \fhi(y_1,y_2,t)\,\mathrm{d}t$ for $(y_1,y_2)\in \R^2$. Next, we claim that the function $v^{11}_A\in C^{2,\alpha}(\T^3)$ defined by
\begin{align*}
v^{11}_A(y):=v^{11}_B(y_1,y_2)\quad\text{for}\quad y=(y_1,y_2,y_3)\in \R^3
\end{align*}
is the solution to the (1,1)-th cell problem corresponding to $A$. Indeed, note that $\int_{[0,1]^3} v^{11}_A = \int_{[0,1]^2} v^{11}_B = 0$ and for any $y=(y_1,y_2,y_3)\in \R^3$ we have that 
\begin{align}\label{-A:D^2v_a^11}
[-A:D^2 v^{11}_A](y) = [-B:D^2 v^{11}_B](y_1,y_2) = b_1(y_1,y_2) - \bar{b}_1 = a_1(y) - \bar{a}_1,
\end{align}
where we have written $\bar{b}_1:=\int_{[0,1]^2} r_B b_1$ and $\bar{a}_1:=\int_{[0,1]^3} r_A a_1$. Therefore, there holds
\begin{align*}
c_1^{11}(A) = \int_{[0,1]^3} r_A\, a_1\, \partial_1 v^{11}_A = \int_{[0,1]^2} r_B\, b_1\, \partial_1 v^{11}_B = c_1^{11}(B) \neq 0,
\end{align*}
and hence, $A$ is type-$\eps$.
\end{proof}

\begin{remark}[type-$\eps$ diagonal constant-trace map in $n=3$]\label{Rk: explicit cbad trace n=3}
We can use Theorem \ref{Thm: construct c-bad trace} to construct a diagonal type-$\eps$ map with constant trace in dimension $n=3$. We choose $b_1,b_2\in C^{\infty}(\T^2;(0,\infty))$ defined by
\begin{align}\label{explicit c-bad n=3 b1}
\begin{split}
b_1(y_1,y_2) &:= \frac{1-\frac{1}{2}\sin(2\pi y_1)\sin(2\pi y_2)}{1+\frac{1}{4}(\cos(2\pi y_1)-2\sin(2\pi y_1))\sin(2\pi y_2)},\\
b_2(y_1,y_2) &:= \frac{1+\frac{1}{2}\sin(2\pi y_1)\sin(2\pi y_2)}{1+\frac{1}{4}(\cos(2\pi y_1)-2\sin(2\pi y_1))\sin(2\pi y_2)}
\end{split}
\end{align}
for $(y_1,y_2)\in\R^2$. Then, $B:=\mathrm{diag}(b_1,b_2)\in C^{\infty}(\T^2;\calS^2_+)$ is type-$\eps$ and we have $c_1^{11}(B) = -\frac{1}{128\pi}\neq 0$; see Remark \ref{Rk: Example recover ST}. Noting that $\|b_1+b_2\|_{\infty} \leq 8$, we define $A:=\mathrm{diag}(a_1,a_2,a_3)\in C^{\infty}(\T^3;\calS^3_+)$ with
\begin{align}\label{explicit c-bad n=3 a1}
a_1(y):= b_1(y_1,y_2),\;\;
a_2(y):= b_2(y_1,y_2),\;\;
a_3(y):= 10 - b_1(y_1,y_2) - b_2(y_1,y_2)
\end{align}
for $y=(y_1,y_2,y_3)\in\R^3$. Then, by Theorem \ref{Thm: construct c-bad trace}, we have that $A$ is type-$\eps$.
\end{remark}

Similarly to the proof of Theorem \ref{Thm: construct c-bad trace}, we find the following result.

\begin{remark}[type-$\eps$ diagonal constant-trace map in $n\geq 4$]\label{Rk: Exp c-bad n>=4}
Let $n\geq 4$ and let $b_1,b_2\in C^{\infty}(\T^2;(0,\infty))$ denote the functions defined in \eqref{explicit c-bad n=3 b1}. Then, the map $A:=\mathrm{diag}(a_1,\dots,a_n)\in C^{\infty}(\T^n;\calS^n_+)$ with $a_1,\dots,a_n\in C^{\infty}(\T^n;(0,\infty))$ defined by
\begin{align*}
a_1(y):=b_1(y_1,y_2),\quad a_2(y):= b_2(y_1,y_2),\quad a_3(y):=10-b_1(y_1,y_2)-b_2(y_1,y_2)
\end{align*}
for $y=(y_1,\dots,y_n)\in\R^n$ and $a_i\equiv 1$ for $4\leq i\leq n$ is type-$\eps$. 
\end{remark}

\subsection{An example with optimal $L^{\infty}$-rate $\calO(\eps)$ and $f\equiv 0$ in $n=3$}\label{Subsec: Example with optimal rate Oeps}

We consider the type-$\eps$ diagonal constant-trace map
\begin{align*}
A=\mathrm{diag}(a_1,a_2,a_3)\in C^{\infty}(\T^3;\calS^3_+)
\end{align*}
defined in Remark \ref{Rk: explicit cbad trace n=3} (see \eqref{explicit c-bad n=3 a1}). Let $\Omega\subset \R^3$ be a bounded smooth domain and let us define the function
\begin{align}\label{explicit g}
g:\bar{\Omega}\rightarrow \R,\quad g(x_1,x_2,x_3):= 8 x_1^3 -3 x_1 x_3^2.
\end{align} 
For $\eps > 0$, we consider the problem
\begin{align}\label{ueps problem f=0}
\left\{\begin{aligned}-A\left(\frac{\cdot}{\eps}\right):D^2 u^{\eps} &= 0& &\text{in }\Omega,\\
u^{\eps} &= g&  &\text{on }\partial\Omega,\end{aligned}\right.
\end{align}
and the corresponding homogenized problem
\begin{align}\label{u problem f=0}
\left\{\begin{aligned}-\bar{A}:D^2 u &= 0& &\text{in }\Omega,\\
u &= g&  &\text{on }\partial\Omega.\end{aligned}\right.
\end{align}
The goal of this section is to prove the following result:
\begin{theorem}[Example with optimal $L^{\infty}$-rate $\calO(\eps)$ and $f\equiv 0$ in $n=3$]
Let $A:=\mathrm{diag}(a_1,a_2,a_3)\in C^{\infty}(\T^3;\calS^3_+)$ with $a_1,a_2,a_3\in C^{\infty}(\T^3;(0,\infty))$ defined in Remark \ref{Rk: explicit cbad trace n=3}, let $\Omega \subset \R^3$ be a bounded smooth domain, and let $g\in C^{\infty}(\bar{\Omega})$ be the function defined in \eqref{explicit g}. Then, there exist constants $C_1,C_2,\eps_0 > 0$ such that 
\begin{align*}
C_1\eps \leq \|u^{\eps}-u\|_{L^{\infty}(\Omega)} \leq C_2\eps\qquad \forall \eps\in (0,\eps_0],
\end{align*}
where $(u^{\eps})_{\eps > 0}$ denotes the sequence of solutions to \eqref{ueps problem f=0} and $u$ denotes the solution to the homogenized problem \eqref{u problem f=0}.
\end{theorem}

\begin{proof}
First, let us recall that the map 
\begin{align*}
B:=\mathrm{diag}(b_1,b_2)\in C^{\infty}(\T^2;\calS^2_+)
\end{align*}
with $b_1,b_2\in C^{\infty}(\T^2;(0,\infty))$ defined in \eqref{explicit c-bad n=3 b1} is type-$\eps$ and we have
\begin{align}\label{cjkl of B values}
c_1^{11}(B) = c_1^{22}(B) = -\frac{1}{128\pi},\quad\quad c_j^{kl}(B) = 0\quad\forall (j,k,l)\not\in\{(1,1,1),(1,2,2)\};
\end{align}
see Remark \ref{Rk: Example recover ST}. Let us denote the invariant measure of $B$ by $r_B\in C^{\infty}(\T^2;(0,\infty))$ and the solution to the (k,l)-the cell problem corresponding to $B$ by $v^{kl}_B\in C^{\infty}(\T^2)$ for $k,l\in\{1,2\}$. 

\medskip
\noindent\textbf{Step 1:} Note that the invariant measure $r\in C^{\infty}(\T^3)$ of $A$ is given by
\begin{align*}
r(y):= r_B(y_1,y_2) = 1+\frac{1}{4}(\cos(2\pi y_1)-2\sin(2\pi y_1))\sin(2\pi y_2)
\end{align*} 
for $y = (y_1,y_2,y_3)\in \R^3$, and that the solutions $v^{11},v^{22},v^{33}\in C^{\infty}(\T^3)$ to the (k,k)-th cell problem corresponding to $A$ for $k\in\{1,2,3\}$ are given by
\begin{align*}
v^{11}(y) := v^{11}_B(y_1,y_2),\quad
v^{22}(y) := v^{22}_B(y_1,y_2),\quad
v^{33}(y) := -v^{11}_B(y_1,y_2)-v^{22}_B(y_1,y_2)
\end{align*}
for $y=(y_1,y_2,y_3)\in \R^3$. Indeed, for $v^{11}$ and $v^{22}$ this is shown as in \eqref{-A:D^2v_a^11}, and for $v^{33}$ we use that $v^{33} = -v^{11}-v^{22}$ to find that
\begin{align*}
-A:D^2 v^{33} = -\left(a_1 - \int_{[0,1]^3} r a_1\right) -\left(a_2 - \int_{[0,1]^3} r a_2\right)= a_3 - \int_{[0,1]^3} r a_3,
\end{align*}
where we have used in the last step that $a_3 = 10 - a_1 - a_2$. 

\medskip
\noindent\textbf{Step 2:} We can now compute the values $c_j^{kl}(A)$ for $j,k,l\in\{1,2,3\}$. First, we note that $c_j^{kl}(A) = 0$ whenever $k\neq l$, and that $c_3^{kk}(A) = \int_{[0,1]^3} r a_3\partial_3 v^{kk} = 0$ for any $k\in \{1,2,3\}$. It remains to find $c_j^{kk}(A)$ for $j\in\{1,2\},k\in \{1,2,3\}$. We have that
\begin{align*}
c_j^{kk}(A) = \int_{[0,1]^3} ra_j\partial_j v^{kk} = \int_{[0,1]^2} r_B\, b_j\, \partial_j v^{kk}_B = c_j^{kk}(B)\qquad\forall j,k\in\{1,2\},
\end{align*}
and, using that $v^{33} = -v^{11}-v^{22}$, we have that
\begin{align*}
c_j^{33}(A) = \int_{[0,1]^3} ra_j\partial_j v^{33} = -\int_{[0,1]^3} ra_j\partial_j v^{11} - \int_{[0,1]^3} ra_j\partial_j v^{22} = -c_j^{11}(A) - c_j^{22}(A)
\end{align*}
for any $j\in\{1,2,3\}$. In view of \eqref{cjkl of B values}, we find that
\begin{align}\label{computed cjkl(A)}
c_1^{11}(A) = c_1^{22}(A) = -\frac{1}{128\pi},\qquad c_1^{33}(A) = \frac{1}{64\pi},\qquad c_j^{kl}(A) = 0\;\text{ when }\; j\neq 1.
\end{align}
\textbf{Step 3:} We claim that the solution $u\in C^{\infty}(\bar{\Omega})$ to the homogenized problem \eqref{u problem f=0} is given by
\begin{align*}
u(x) := g(x) = 8 x_1^3 -3 x_1 x_3^2
\end{align*}
for any $x=(x_1,x_2,x_3)\in \bar{\Omega}$. Note that the effective coefficient $\bar{A}\in \calS^3_+$ is given by
\begin{align*}
\bar{A} := \int_{[0,1]^3} rA = \mathrm{diag}(1,1,8)
\end{align*}
and that $u$ satisfies $-\bar{A}:D^2 u = -\partial_{11}^2 u -\partial_{22}^2 u - 8\,\partial_{33}^2 u \equiv 0$. As also $u=g$ in $\bar{\Omega}$, we have that $u$ is the solution to \eqref{u problem f=0}.

\medskip
\noindent \textbf{Step 4:} Let $z$ denote the unique solution to \eqref{z problem}, and let us recall from \eqref{Linfty bound} that 
\begin{align}\label{recall 1.7 for St4}
\left\|u^{\eps} - u + 2\eps z\right\|_{L^{\infty}(\Omega)} = \calO(\eps^2)\quad\text{as }\eps \searrow 0.
\end{align}
We show that $z\not\equiv 0$. To this end, we use \eqref{computed cjkl(A)} to obtain that
\begin{align*}
\sum_{j,k,l=1}^3 c_j^{kl}(A)\,\partial_{jkl}^3 u = -\frac{1}{128\pi}\partial_{111}^3 u -\frac{1}{128\pi}\partial_{122}^3 u +\frac{1}{64\pi}\partial_{133}^3 u \equiv -\frac{15}{32\pi}.
\end{align*}
In view of \eqref{z problem}, it follows that $z\not\equiv 0$, and we conclude from \eqref{recall 1.7 for St4} that there exist constants $C_1,C_2,\eps_0 > 0$ such that 
\begin{align*}
C_1\eps \leq \|u^{\eps}-u\|_{L^{\infty}(\Omega)} \leq C_2\eps\qquad \forall \eps\in (0,\eps_0].
\end{align*}
The proof is complete.
\end{proof}

\subsection{type-$\eps^2$ map $A$ for which $A+I_n$ is type-$\eps$ in $n=2$}\label{Subsec: A+I}

In this section, we provide an example of a map $A\in C^{\infty}(\T^2;\calS^2_+)$ such that $A$ is type-$\eps^2$ and $A+I_2$ is type-$\eps$, thus demonstrating Remark \ref{Rk: A+I}. We define $A\in C^{\infty}(\T^2;\calS^2_+)$ by
\begin{align}\label{A for A+I}
A(y):=\frac{1}{r(y)}\, \mathrm{diag}(a(y),2-a(y))\quad\text{for}\quad y\in \R^2,
\end{align}
where $r\in C^{\infty}(\T^2;(0,\infty))$ and $a\in C^{\infty}(\T^2)$ are given by
\begin{align}\label{A for A+I ra}
\begin{split}
r(y_1,y_2)&:= 1 + \frac{1}{3} \sin(2\pi(y_1+y_2)) + \frac{1}{3} \cos(2\pi(y_1 - y_2)),\\
a(y_1,y_2)&:= 1 + \frac{1}{2}\sin(4\pi(y_1+y_2))
\end{split}
\end{align}
for $(y_1,y_2)\in \R^2$.

\subsubsection{The map $A$ defined in \eqref{A for A+I}-\eqref{A for A+I ra} is type-$\eps^2$}

Noting that $r$ is the invariant measure of $A$, we can use Lemma \ref{Thm: r,a special} to show that $A$ is type-$\eps^2$. Note that 
\begin{align*}
r(y_1,y_2) = r_1(y_1+y_2)  +  r_2(y_1 - y_2)
\end{align*}
for any $(y_1,y_2)\in \R^2$, where $r_1,r_2\in C^{\infty}(\T)$ are the functions given by
\begin{align*}
r_1(t):= 1 + \frac{1}{3}\sin(2\pi t),\quad r_2(t):=\frac{1}{3}\cos(2\pi t)\quad\text{for}\quad t\in \R.
\end{align*}
Then, the functions $R_i \in C^{\infty}(\T)$ satisfying $R_i'' = r_i - \int_{(0,1)} r_i$ and $\int_{(0,1)} R_i = 0$ for $i\in\{1,2\}$ are given by $R_1(t) := -\frac{1}{12\pi^2}\sin(2\pi t)$ and $R_2(t):=-\frac{1}{12\pi^2}\cos(2\pi t)$ for $t\in \R$. Noting that
\begin{align*}
\int_0^1 \int_0^1 a(y_1,y_2) R_1'(y_1+y_2)\,\mathrm{d}y_1\,\mathrm{d}y_2 = \int_0^1 \int_0^1 a(y_1,y_2) R_2'(y_1-y_2)\,\mathrm{d}y_1\,\mathrm{d}y_2 = 0,
\end{align*}
we deduce from Lemma \ref{Thm: r,a special} that $A$ is type-$\eps^2$.

\subsubsection{The map $A+I_2$ with $A$ defined in \eqref{A for A+I}-\eqref{A for A+I ra} is type-$\eps$} 

Numerically, we see that for $\tilde{A} \in C^{\infty}(\T^2;\calS^2_+)$ given by
\begin{align*}
\tilde{A}(y) := A(y)+I_2 \quad\text{for}\quad y\in \R^2,
\end{align*}
we have that $c_1^{11}(\tilde{A})=0.0005\dots \neq 0$ and thus, $\tilde{A}$ is type-$\eps$. The invariant measure and the effective coefficient have been approximated by the finite element scheme presented in \cite{CSS20} and the corrector by a standard $W_{\mathrm{per}}(Y)$-conforming finite element scheme.

\subsection{type-$\eps$ diagonal map $A$ with $r\equiv 1$ in $n=2$}\label{Subsec: c-bad diagonal r=1 n=2}

In this section, we provide an example of a diagonal map $A\in C^{\infty}(\T^2;\calS^2_+)$ which has invariant measure $r\equiv 1$ and is type-$\eps$. We define $A\in C^{\infty}(\T^2;\calS^2_+)$ by
\begin{align}\label{A c-bad r=1 in n=2}
A(y):=\mathrm{diag}(a_1(y),a_2(y))\quad\text{for}\quad y\in \R^2,
\end{align}
where $a_1,a_2\in C^{\infty}(\T^2;(0,\infty))$ are defined by
\begin{align}\label{A c-bad r=1 in n=2 2}
\begin{split}
a_1(y_1,y_2)&:= 1 - \frac{1}{2}\sin(2\pi(y_1+y_2)) + \frac{1}{4}\sin(4\pi y_2),\\
a_2(y_1,y_2)&:= 1 + \frac{1}{2}\sin(2\pi(y_1+y_2)) + \frac{1}{4}\cos(2\pi y_1)
\end{split}
\end{align}
for $(y_1,y_2)\in \R^2$. We have $D^2:A = \partial_{11}^2 a_1 + \partial_{22}^2 a_2\equiv 0$ and thus, the invariant measure of $A$ is given by $r\equiv 1$. The effective coefficient to $A$ is given by $\bar{A}:=\int_Y rA = I_2$. Numerically, we see that $c_1^{11}(A) = -0.00001\dots \neq 0$ and hence, $A$ is type-$\eps$. We have used a standard $W_{\mathrm{per}}(Y)$-conforming finite element scheme to approximate $v^{11}$ and omit the details. 

As a consequence, we find that there exists a type-$\eps^2$ diagonal $\tilde{A}\in C^{\infty}(\T^2;\calS^2_+)$ for which $\tilde{r}\tilde{A}$ is type-$\eps$, where $\tilde{r}$ denotes the invariant measure to $\tilde{A}$. Indeed, let us define the map $\tilde{A}\in C^{\infty}(\T^2;\calS^2_+)$ by
\begin{align*}
\tilde{A}(y):=\frac{1}{a_1(y)}A(y) = \mathrm{diag}\left(1,\frac{a_2(y)}{a_1(y)}\right)\quad\text{for}\quad y\in \R^2,
\end{align*}
with $a_1,a_2\in C^{\infty}(\T^2;(0,\infty))$ and $A\in C^{\infty}(\T^2;\calS^2_+)$ defined in \eqref{A c-bad r=1 in n=2}-\eqref{A c-bad r=1 in n=2 2}. Then, $\tilde{A}$ is type-$\eps^2$ by Theorem \ref{Thm: Main}. Noting that the invariant measure $\tilde{r}\in C^{\infty}(\T^2;(0,\infty))$ of $\tilde{A}$ is given by $\tilde{r} = a_1$, we find that $\tilde{r}\tilde{A} = A$ is type-$\eps$.

\section{Conclusion and Open Questions}\label{Sec: Conc}

Let us briefly review what has been achieved in this work. We started by showing that any diffusion matrix $A$ which can be written as $A=C+aM$ for some $C, M\in \R^{n\times n}_{\mathrm{sym}}$ and $a\in C^{0,\alpha}(\T^n)$ is type-$\eps^2$. From this, we have deduced two interesting consequences in dimension $n=2$: the fact that diagonal diffusion matrices with constant trace are type-$\eps^2$, and the fact that for any diagonal diffusion matrix the $A(\frac{\cdot}{\eps})$-harmonic functions homogenize at rate $\calO(\eps^2)$. Both of these statements fail to hold in dimensions $n\geq 3$. 

The second major result of this work was to fully characterize the set of type-$\eps^2$ diagonal diffusion matrices in dimension $n=2$. 

As a third major result, we have found a way of constructing families of type-$\eps^2$ diffusion matrices of the form $\gamma A$ from a given type-$\eps^2$ diagonal $A$, where $\gamma$ is some positive scalar function depending on $A$. 

The fourth major result of this paper was to prove that the set of type-$\eps$ diffusion matrices in $C^{0,\alpha}(\T^n;\calS^n_+)$ is open and dense in $C^{0,\alpha}(\T^n;\calS^n_+)$ when $n\geq 2$. Next to these four aforementioned major results, we have further analyzed various aspects regarding type-$\eps^2$ and type-$\eps$ diffusion matrices, which we will not recall here for the sake of brevity. 

The following list of questions provides a series of open problems for future work:

\begin{question}
For $n\in \N$ and $\alpha\in (0,1)$, we define the set $\mathcal M_n^{\alpha} \subset C^{0,\alpha}(\T^n;\calS^n_+)$ by $\mathcal M_n^{\alpha}:=\{A\in C^{0,\alpha}(\T^n;\calS^n_+):\mathrm{tr}(A) \equiv 1\}$. Suppose $n\geq 3$ and $\alpha\in (0,1)$. Is it true that the set $\{A\in \mathcal{M}_n^{\alpha}: A\text{ is type-$\eps$}\}$ is open and dense in $\mathcal{M}_n^{\alpha}$?
\end{question}

\begin{question}
Find a complete characterization of the set of type-$\eps^2$ diagonal maps in $C^{0,\alpha}(\T^3;\calS^3_+)$, similarly to Theorem \ref{Thm: Complete char} in the two-dimensional case.
\end{question}

\begin{question}
Let $A\in C^{0,\alpha}(\T^n;\calS^n_+)$ with $\alpha\in (0,1)$ be a type-$\eps^2$ diagonal diffusion matrix. By Theorem \ref{Thm: 1/C:A A}, if $C\in \R^{n\times n}$ is such that $\gamma:=\frac{1}{C:A}\in C^{0,\alpha}(\T^n;(0,\infty))$, then $\gamma A$ is type-$\eps^2$.
Are there other generic ways to construct $\theta \in C^{0,\alpha}(\T^n;(0,\infty))$ so that $\theta A$ is type-$\eps^2$?
\end{question}

\begin{question}
Do the results of Theorem \ref{Thm: Main} and Conjecture \ref{Conj} remain true in the framework of stochastic homogenization?
\end{question}

\section*{Acknowledgments}

The authors thank Professor Hongjie Dong (Brown University) for useful discussions on the regularity of invariant measures, Professor Yves Capdeboscq (Universit\'e de Paris) for helpful conversations during the preparation of this work, and Professor Scott Armstrong (New York University) for useful discussions on the third-order homogenized tensor and suggestions on the writing. The work of HT is partially supported by NSF CAREER grant DMS-1843320 and a Simons Fellowship. The work of XG is supported by Simons Foundation through Collaboration Grant for Mathematicians \#852943.

\bibliographystyle{plain}
\bibliography{ref_GST.bib}

\end{document}